\documentclass[final]{siamonline250211}
\usepackage{tabularx}
\usepackage{mathtools}
\usepackage{amssymb}
\usepackage{booktabs}
\usepackage{array}
\usepackage{enumitem}
\setlist[enumerate]{leftmargin=2.5em}
\usepackage{graphicx}
\usepackage{algorithm}
\usepackage{algorithmic}
\usepackage{placeins}
\usepackage{changepage}
\usepackage{xcolor}
\usepackage[labelstoglobalaux]{bibunits}
\defaultbibliographystyle{siamplain}
\usepackage{macros}
\usepackage{pgfplots}
\usepgfplotslibrary{groupplots}
\pgfplotsset{compat=1.18}
\newsiamremark{remark}{Remark}
\title{Plug-and-Play Methods Provably Converge Even with Improperly Trained Denoisers: Convergence by Architectural Design\thanks{The authors acknowledge support from NVIDIA Corporation through the NVIDIA Academic Grant Program for the donation of the RTX PRO 6000 Blackwell Max-Q Workstation Edition GPUs used in this research as well as the support from a Hellman Fellowship from the University of California.}}
\author{Henry Pritchard\thanks{The authors are with the Department of Electrical and Computer Engineering, University of California, San Diego (\email{hepritchard@ucsd.edu},
\email{rahul@ucsd.edu}).} \and Rahul Parhi\footnotemark[2]}
\begin{document}
\begin{bibunit}[siamplain]

\newlist{hypotheses}{enumerate}{1}
\setlist[hypotheses]{
    label=\normalfont S\arabic*.,
    ref=S\arabic*,
    leftmargin=3em,
    labelsep=0.5em,
    align=left
}
\crefformat{hypothesis}{(#2#1#3)}
\Crefformat{hypothesis}{(#2#1#3)}
\maketitle

\begin{abstract}
Plug-and-Play (PnP) methods replace proximal operators with learned denoisers, which produce state-of-the-art reconstruction quality, but sacrifice the variational interpretation and convergence guarantees of proximal methods. Learned Proximal Networks (LPNs) address this problem by designing the denoiser architecture so that it is exactly the proximal operator of a regularizer. In this work, we clarify the theoretical foundations of LPNs and extend the framework to a broader class of activation functions. We then study two practical mechanisms for controlling the learned prior: (i) averaging the LPN with the identity, a common heuristic in PnP methods, and (ii) directly scaling the implicit regularizer induced by the proximal operator. In particular, we characterize the regularizer induced by averaging and develop a convergent method for evaluating the scaled proximal operator.
\end{abstract}

\section{Introduction}\label{section:introduction}
Inverse problems are ubiquitous in imaging, signal processing, and statistics. Given a \emph{ground truth} $\bar{x} \in \R^n$, \emph{forward operator} $A:\R^n \to \R^m$, and \emph{measurement noise} $\epsilon \in \R^m$, the \emph{observation} $y \in \mathbb{R}^m$ is obtained via
\begin{align}\label{eq:inverse_problem}
    y = A\bar{x} + \epsilon.
\end{align}
The goal is to recover an estimate of $\bar{x}$, which we call $\hat{x}$, from $y$. For most inverse problems of interest, $A$ is underdetermined, so the least squares solution is typically unstable. For this reason, the problem is typically posed as the minimization of a \emph{regularized objective}:\begin{align}\label{eq:regularized_objective}
    \hat{x} \in \argmin_{x\in \R^n} F(x), \quad F(x) \coloneqq \eta \fid(x) + \varphi(x)
\end{align}
where $\fid: \R^n \to \R$ is a \emph{data-fidelity term}, $\varphi:\R^n \to \R$ is a \emph{regularizer}, and $\eta>0$ balances the two. The data fidelity term encourages $A\hat{x} \approx y$ in an appropriate sense, often taken to be least-squares loss $\fid(x) = \tfrac{1}{2}\|Ax-y\|^2$. The regularizer encourages prior beliefs about the class of \emph{admissible solutions}. For instance, the $\ell^1$ norm, $\varphi(x) = \|x\|_1$, promotes sparsity and underlies much of the theory of compressed sensing~\cite{candes2006robust,donoho2006compressed, mesforush2026lasso}. Total Variation (TV), $\mathrm{TV}(x) \coloneqq\sum_{i=1}^{n-1} |x_{i+1} - x_i|$, encourages solutions to be piece-wise constant and is widely used in image denoising~\cite{chambolle2004algorithm,marquina2000explicit, beck2009fast}.

The proximal operator (\cref{def:proximal_operator}) is a central object for solving~\cref{eq:regularized_objective}. Since gradient steps are not appropriate for nonsmooth regularizers, the proximal operator, $\prox_\varphi$ is used to handle $\varphi$ implicitly. $\prox_\varphi$ is a denoising step, mapping corrupted inputs to nearby signals that better agree with the prior $\varphi$, which motivates  extensions of proximal algorithms, where the proximal operator is replaced by a denoiser.

In this work, we revisit Learned Proximal Networks (LPNs), introduced in~\cite{fang2024whats}, which use a simple architectural constraint to guarantee that a learned denoiser is a proximal operator. This places the resulting plug-and-play methods back within the classical framework of proximal methods. Unlike approaches whose convergence guarantees rely on unverifiable assumptions about the denoiser, training heuristics, or an idealized denoising model, the LPN proximal structure is enforced exactly by the architecture. Consequently, this guarantee is largely independent of the training procedure, including the choice of training data, loss function, and optimization algorithm. We build on this foundation by clarifying the underlying theory, extending the framework to a broader class of architectures, and establishing analogous guarantees for averaged LPNs, which interpolate the learned proximal map with the identity, and scaled LPNs, which realize proximal operators associated with scaled versions of the learned regularizer.

\subsection{A Bayesian Perspective}\label{subsection:bayesian_perspective}
While the regularized objective in~\cref{eq:regularized_objective} is meaningful from a variational perspective, it also admits a natural probabilistic interpretation. If $\bar{x}$ is a random variable with density $p$, related to the observation $y$ through likelihood $p(y\mid \cdot )$, the \emph{maximum a posteriori} (MAP) estimate is
\begin{align}\label{eq:map_estimator}
    \hat{x}_{\mathrm{MAP}}(y) \in \argmin_{x\in\R^n}
     -\log p(y\mid x) - \log p(x).
\end{align}
This is exactly the regularized objective in~\cref{eq:regularized_objective} with data-fidelity term $\eta \fid(x) = -\log p(y\mid x)$ and regularizer $\varphi(x) = -\log p(x)$. In particular, the least-squares data-fidelity term is realized when the measurement noise is assumed to be Gaussian. 

The same probabilistic view applies to the proximal operator (\cref{def:proximal_operator}). If $p(x) \propto \mathrm{e}^{-\gamma \varphi(x)}$, and $y$ is obtained via~\cref{eq:inverse_problem} with measurement operator $A=I_n$, under measurement noise $\epsilon \sim \mathcal{N}(0,\eta I_n)$, then the MAP estimator of $x$ is 
\begin{align}
    \hat{x}_{\mathrm{MAP}}(y) \in \argmin_x\left\{\frac{1}{2\eta}\|x-y\|^2 + \gamma\varphi(x)\right\} = \prox_{\gamma \eta \varphi}(y).
\end{align} Therefore, the proximal operator is a MAP denoising step under a prior proportional to $\mathrm{e}^{-\gamma \varphi(x)}$, with additive Gaussian measurement noise of variance $\eta$. Accordingly, the effective prior strength is controlled by $\gamma\eta$. Larger $\eta$ means heavier measurement noise which increases the reliance on the prior, while larger $\gamma$ corresponds to a stronger prior.

\subsection{Proximal Gradient Descent} \label{subsection:pgd}
Gradient Descent (GD) is a classical technique for minimizing \emph{smooth} (Lipschitz gradient) objective functions, with iterations
\begin{align}\label{eq:gradient_descent}
    x_{k+1} \coloneqq x_k - \gamma \nabla F(x_k), \tag{GD}
\end{align}
for a step size $\gamma>0$. When $F$ is smooth and bounded below, the objective values $F(x_k)$ converge for an appropriately chosen step size. However, objectives of the form~\cref{eq:regularized_objective} often involve nonsmooth regularizers, making gradient descent inappropriate. Typically, the data-fidelity term $\fid$ is smooth with a straightforward gradient, while the regularizer $\varphi$ is not. Proximal Gradient Descent (PGD) exploits the structure by taking a gradient step on $\fid$, followed by a proximal step on $\varphi$. The resulting iteration is 
\begin{align}\label{eq:pgd}\tag{PGD}
    x_{k+1} \coloneqq \prox_{\gamma \varphi} (x_k -   \gamma\eta \nabla \fid (x_k)),
\end{align}
where $\gamma > 0$ is the step size and $\eta > 0$ is the balancing term from the objective. See~\cite{combettes2011proximal,beck2009fast,attouch2013convergence} for classical convergence results for PGD.

\subsection{Plug-and-Play Methods} \label{subsection:PnP}
Proximal operators are essentially denoisers that favor a certain prior, as discussed in~\cref{subsection:bayesian_perspective}. It is therefore natural to replace them with a more powerful trained denoiser, which we denote $D_\sigma$, that need not arise from any explicit regularizer.\footnote{The $\sigma$ in $D_\sigma$ signifies that the denoiser should appropriately match the level of noise in~\cref{eq:inverse_problem}.} This substitution yields Plug-and-Play (PnP) methods, introduced in~\cite{venkatakrishnan2013plug}.\footnote{PnP was originally introduced in the context of ADMM~\cite{venkatakrishnan2013plug}, but the term is now used more broadly for methods that replace a proximal operator with an off-the-shelf denoiser.} For example, the PnP variant of the PGD iteration is given by 

\begin{align}\label{eq:PnP-PGD}\tag{PnP-PGD}
    x_{k+1} \coloneqq D_\sigma (x_k -  \eta \nabla \fid (x_k)).
\end{align}
Although PnP methods often achieve excellent empirical performance, the denoiser is not generally the proximal operator of any explicit regularizer. Consequently a PnP algorithm is generally not associated with the minimization of an explicit objective $F$, and loses the variational interpretation that motivates classical methods. 

 Without additional assumptions, a PnP algorithm's behavior can be assessed only by inspecting the iterates, which may give little insight into whether the algorithm is making meaningful progress or becoming unstable. In particular, oscillatory or non-monotone behavior may reflect divergence, convergence to an undesirable solution, or merely transient behavior along a convergent trajectory. A convergence guarantee with an explicit rate provides a principled way to assess progress in such an algorithm. Consequently, a substantial body of work studies conditions under which PnP methods converge. As we discuss in the following section, however, the resulting guarantees that currently exist in the literature are unsatisfying. Indeed, they generally require conditions on $D_\sigma$ that cannot be verified in practice. 

\subsection{Existing Methods}\label{subsection:existing_methods}
The central challenge in theoretical PnP analysis is to identify conditions on the denoiser $D_\sigma$ and data-fidelity term $\fid$ that are both (1) sufficient for convergence and (2) enforceable in practice without sacrificing performance. Much of the existing theory focuses primarily on the first requirement. These works establish convergence under conditions that are either impossible to verify in practice, enforced only approximately through training heuristics, or achieved through constraints that can substantially limit denoiser expressivity. We focus in particular on the learned-denoiser setting, where $D_\sigma$ is represented by a deep neural network (DNN).

One popular line of work seeks to constrain the Lipschitz constant of a neural network. Gradient-step denoisers~\cite{hurault2021gradient,hurault2022proximal,hurault2024convergent}, nonexpansive or averaged denoisers~\cite{terris2020building, ryu2019plug, liu2021recovery, sun2021scalable, sun2019online}, monotone operators~\cite{pesquet2021learning}, and deep equilibrium models~\cite{gilton2021deep} each develop distinct convergence theory, but all require control of the Lipschitz constant of a neural network. Most activations are already $1$-Lipschitz, so the problem reduces to controlling the Lipschitz constant of large weight matrices and convolutional layers. In theory, this condition can be enforced by applying spectral normalization~\cite{ryu2019plug} or Bj\"orck orthonormalization~\cite{anil2019sorting} to each layer of the network. However, these are iterative procedures that enforce the desired constraint only in the limit. Running them to convergence throughout training for each layer is prohibitively costly, so they are generally truncated, yielding a method that amounts to a heuristic, rather than a guarantee. Further, we note that just \emph{computing} the Lipschitz constant of even a two-layer network is NP-hard~\cite{virmaux2018lipschitz}, leaving little hope of exactly controlling it for modern deep denoisers.

Another line of work, known as Regularization by Denoising (RED)~\cite{romano2017little}, is based on the observation that the noise removed from a signal should be approximately independent of the signal, i.e.,  $\langle x-D_\sigma(x),x\rangle \approx 0$ for an effective denoiser. This motivates the RED regularizer as $\rho_{\mathrm{RED}}(x) \coloneqq \tfrac{1}{2}\langle x-D_\sigma(x),x\rangle$. Under suitable assumptions on $D_\sigma$, $\rho_{\mathrm{RED}}$ is convex and has a gradient that is simple to evaluate: $\nabla \rho_{\mathrm{RED}}(x) = x-D_{\sigma}(x)$. Convergence of PGD and ADMM schemes is therefore nearly immediate from existing theory. 

However, the assumptions of the RED framework are even more restrictive than the Lipschitz assumptions above. Indeed, strong passivity and Jacobian symmetry ($JD_\sigma(x)=JD_\sigma(x)^T$) imply that $D_\sigma$ is 1-Lipschitz. Further, Jacobian symmetry is essential, as the authors of~\cite{reehorst2018regularization} showed that there is no function $\rho$ such that $\nabla \rho(x) = x-D_\sigma(x)$ if $JD_\sigma(x)$ is not symmetric. They further showed that the most common denoisers fail to have symmetric Jacobians. Therefore, RED provides an elegant variational interpretation for only an extremely restricted class of denoisers. Later, the RED framework was extended to a broader class of denoisers in~\cite{cohen2021regularization}, but this method still requires properties of the denoiser that are plausible but unverifiable in practice, namely demicontractivity and certain fixed-point assumptions.

Other convergence results avoid these Lipschitz-type restrictions entirely. For example, rigorous convergence theory exists for linear and kernel denoisers~\cite{gavaskar2020plug,gavaskar2021plug}, for which the relevant properties are verified by construction. However, restricting the denoiser to a linear operator substantially limits its expressivity relative to modern learned denoisers. At the opposite extreme, convergence has been established for exact MMSE denoisers~\cite{xu2020provable,pritchard2026nonasymptotic}, which can represent much richer image priors but are generally only \emph{approximated} by a learned model. Finally, fixed-point convergence has been established under a bounded-residual assumption~\cite{chan2016plug}, requiring $\|D_\sigma(x)-x\|^2 \leq n\sigma^2C$ for some $C$ independent of $n$ and $\sigma$. This global property is milder, but still cannot be verified in practice.

\subsection{Contributions}\label{subsection:contributions}
In this work, we revisit the Learned Proximal Network (LPN) framework introduced in~\cite{fang2024whats}, in which the denoiser used within a PnP method is constrained to be a proximal operator. Given an input convex neural network (ICNN) $\psi_\theta:\R^n \to \R$ and a strong convexity parameter $0<\alpha<1$, we define the LPN to be $\lpn(x)\in\partial\potential(x)$, where $\potential \coloneqq \psi_\theta+\frac{\alpha}{2}\|\cdot\|^2$ is the \emph{potential function} and $\partial \potential$ is its convex subdifferential (\cref{def:subgradient}). Classical characterizations of proximal operators then guarantee that
\begin{equation}
\lpn(x)\in\prox_{\reg}(x),
\end{equation}
where $\reg\coloneqq\potential^*-\tfrac{1}{2}\|\cdot\|^2$ is a data-driven regularizer implicit in the LPN. Under mild assumptions on $\psi_\theta$ that are satisfied by all ICNNs built with standard activations, $\reg$ is continuously differentiable, coercive, and satisfies the KL property~\cref{eq:kl_inequality}. The desired convergence results therefore follow from the classical framework of proximal methods. Our contributions build on this framework:

\paragraph{Clarifying the theoretical foundations of LPNs}
In~\cref{section:preliminaries} and~\cref{section:LPNs}, we clarify the mathematical foundations of the LPN construction and show that its proximal interpretation follows directly from classical results.

\paragraph{LPNs using a broader class of activation functions}
Whereas the original LPN formulation focuses on softplus activations, in~\cref{section:LPNs} we extend the framework to a substantially broader class that includes non-differentiable activations. This includes nonsmooth activations such as ReLU, as well as smooth alternatives such as Huberized ReLU. We establish convergence of the resulting LPN-PGD algorithm~\cref{alg:lpn_pgd_vanilla} in~\cref{prop:pgd_convergence}.

\paragraph{A variational interpretation of denoiser averaging}
When a denoiser is too strong, a common approach in PnP methods is to average it with the identity. Given an LPN $\lpn$ and $\lambda\in(0,1)$, we show that $\lpn^{(\lambda)} \coloneqq (1-\lambda)I+\lambda\lpn$ still falls within the LPN framework in~\cref{prop:closure_of_lpns}, and characterize precisely how this averaging operation relaxes the underlying learned prior and the associated PnP algorithm in~\cref{cor:averaged_lpn}.

\paragraph{Direct scaling of the learned regularizer}
In~\cref{subsection:regularizer_scaling}, we present a method for directly adjusting the strength of the prior encoded by an LPN by evaluating $\prox_{\gamma\reg}$ for $0<\gamma\leq1$. Although an LPN directly evaluates only $\prox_{\reg}$, we show that $\prox_{\gamma\reg}(x)$ can be recovered by solving a strongly convex auxiliary problem using only evaluations of the original LPN. We then develop a method that permits the auxiliary problem to be approximately solved, while still converging to a critical point of the scaled objective through the machinery of inexact proximal gradient descent (\cref{alg:lpn_pgd_scaled}). The resulting convergence guarantee is given in~\cref{prop:scaled_pgd_convergence}.

\paragraph{Full-image LPNs for MRI and CT reconstruction}
In~\cref{section:experiments}, we demonstrate the resulting framework on accelerated multicoil MRI and on low-dose and sparse-view computed tomography inverse problems. This departs from the patchwise implementation used for the large-scale MayoCT experiments in \cite{fang2024whats}, where a $128\times128$ LPN was applied to overlapping windows of a larger image and the reconstructed patches were averaged. Crucially, the patchwise application of an LPN is neither an LPN nor a proximal operator; it therefore discards the architectural guarantee on which the convergence theory rests. By instead evaluating a single full-image LPN, we retain the exact proximal interpretation guaranteed by the theory. We present results for ordinary proximal gradient descent, as well as the averaged and scaled approaches.

\section{Preliminaries}\label{section:preliminaries}
In this section, we collect select definitions and results that are used in the remainder of the paper. We restrict attention to $\R^n$, although many of the results hold in greater generality. $\|\cdot\|$ and $\langle \cdot,\cdot\rangle$ always denote the Euclidean norm and
inner product. We begin with the proximal operator, the main focus of this work.

\begin{definition}[Proximal operator]\label{def:proximal_operator}
Let $h: \R^n \to \widebar{\R}$ be proper and lower-semicontinuous with $\gamma>0$. The \emph{proximal operator} of $h$ is
\begin{align}\label{eq:proximal_operator}
    \prox_{\gamma h} (x) \coloneqq \argmin_{z\in \R^n}\left\{h(z)+\frac{1}{2\gamma}\|z-x\|^2\right\}.
\end{align}
\end{definition}
The proximal operator balances decreasing $h$ with remaining close to the current point. It is related to a standard gradient descent step in the following way. For a differentiable and convex function $h:\R^n \to \R$, gradient flow is defined to be \begin{equation}\label{eq:gradient_flow}
    \dot{x}(t) = -\nabla h(x(t)), \quad t\geq 0. \tag{GF} 
\end{equation}
One can show convergence of~\cref{eq:gradient_flow} under mild conditions, aided primarily by the fact that $h(x(t))$ is non-increasing along the trajectory, since \begin{equation}
    \frac{d}{dt}h(x(t)) = \langle \nabla h(x(t)),\dot{x}(t)\rangle  = -\|\nabla h(x(t))\|^2 \leq 0.
\end{equation}
A \emph{standard} gradient step is obtained by applying a forward Euler discretization to~\cref{eq:gradient_flow}:\begin{equation}
    \frac{x_{k+1}-x_k}{\gamma} = -\nabla h(x_k),
\end{equation}
where $\gamma>0$ is the step size. This is an explicit step: the gradient is evaluated at the current iterate $x_k$. In contrast, a proximal operator is obtained from a backward Euler discretization:
\begin{equation}\label{eq:prox_optimality}
    \frac{x_{k+1}-x_k}{\gamma} = -\nabla h(x_{k+1}).
\end{equation}
Rearranging gives $\nabla h(x_{k+1}) + \tfrac{1}{\gamma}(x_{k+1}-x_k) = 0$, the first order optimality condition for~\cref{eq:proximal_operator}, meaning $x_{k+1} = \prox_{\gamma h} (x_k)$. Because a proximal step evaluates the gradient implicitly at the future iterate, the update accounts for the geometry of the objective at the next iterate. This implicit correction often makes proximal methods more stable than explicit gradient methods. Proximal operators can also be interpreted in terms of Moreau envelopes, defined next.
\begin{definition}[Moreau envelope]
    Take $h:\R^n \to \R$ with $\gamma>0$. The Moreau Envelope of $h$ is given by\begin{equation}
        M_{\gamma}h (y) \coloneqq \inf_x\left\{h(x) + \frac{1}{2\gamma}\|x-y\|^2 \right\}.
    \end{equation}
\end{definition}
When $\prox_{\gamma h}$ is single-valued and continuous, it can be written as a gradient step on $M_\gamma h$: $\prox_{\gamma h}(x) = x-\gamma \nabla M_\gamma h(x)$.\footnote{\cite[Theorem II.6]{pritchard2026nonasymptotic} proves this exact formulation, but this identity is well-known in the literature under various assumptions on $h$.} In this way, a proximal operator is a \emph{smoothed} gradient step. We next recall the convex conjugate.

\begin{definition}[Convex conjugate] Let $\Psi:\R^n\to\widebar{\R}$ be proper and convex. The \emph{convex conjugate} $\Psi^*$ is 
\begin{equation}\label{eq:convex_conjugate} \Psi ^*(y) \coloneqq \sup_{x\in\R^n} \left\{ \langle x,y\rangle - \Psi(x) \right\}. \end{equation} 
\end{definition} 
The convex conjugate can be defined for a wide variety of functions, not just convex ones. When it is well-defined, $\Psi^*$ is always convex, even when $\Psi$ is not. Additionally, when $\Psi$ is convex, proper, and lower semicontinuous, the Fenchel-Moreau theorem states that biconjugation recovers the original function, i.e., $(\Psi^*)^* = \Psi$~\cite{rockafellar1997convex}[Theorem 12.2]. Another useful identity we make use of is that for strongly convex and differentiable $\Psi$,
\begin{equation}\label{eq:conjugate_gradient_identity}
    \nabla \Psi^* = (\nabla \Psi)^{-1}.
\end{equation}
See, for example,~\cite[Theorem 26.5]{rockafellar1997convex}. We next define the subgradient and subdifferential, a generalization of the gradient for convex functions. 

\begin{definition}[Subgradient]\label{def:subgradient}
Let $\Psi:\R^n\to\widebar{\R}$ be proper and convex, and let $x\in\dom(\Psi)$. A vector $g\in\R^n$ is a (convex) \emph{subgradient} of $\Psi$ at $x$ if \begin{equation}
\Psi(y)
\geq
\Psi(x)+\langle g,y-x\rangle
\quad \forall y\in\R^n.
\end{equation}
The set of all such vectors is called the (convex) \emph{subdifferential} at $x$, denoted $\partial \Psi(x)$. 
\end{definition}
The convex subdifferential, which we simply refer to as the subdifferential, is the only notion of subdifferential needed in this work. It also admits a simple geometric interpretation: $\partial \Psi(x) $ contains the slopes of all planes that support $\Psi$ at $x$. The convex subdifferential is always a convex set, and when $\Psi$ is differentiable at $x$, $\partial\Psi(x)$ simply reduces to the singleton $\{\nabla \Psi(x)\}$. The subdifferential can also be generalized to a wider class of nonconvex functions. For example, the Fréchet variant $\hat{\partial} \Psi(x)$ extends the notion of a differential to the slopes of all planes that \emph{locally} support $\Psi$ at $x$. See~\cite[Chapter 8]{rockafellar1998variational} for more detail. We next present a result that relates proximal operators to the subdifferentials of convex functions. This connection was first made in~\cite{gribonval2020characterization}.

\begin{lemma}\label{lem:coercive_regularizer}
Consider $f:\R^n \to \R^n$ and convex lower semi-continuous $\Psi:\R^n \to \R$. Then if $f(x) \in \partial \Psi(x)$ for all $x \in \R^n$, then $f(x) \in \prox_{\varphi}(x)$ for all $x\in \R^n$, where 
\begin{equation}\label{eq:bar_phi}
    \varphi \coloneqq \Psi^* - \frac{1}{2}\|\cdot\|^2.
\end{equation}
Further, if there exists some $\alpha \in (0,1)$ and finite $K,L$ such that $\Psi(x) \leq \tfrac{\alpha}{2}\|x\|^2 + L\|x\| + K$ for all $x\in \R^n$, then $\varphi$ is coercive.
\end{lemma}
\begin{proof}
From the definition, 
    \begin{align}
        \prox_\varphi(y) 
        = \argmin_{x}\left\{\Psi^*(x) -\langle x,y\rangle +\frac{1}{2}\|y\|^2  \right\} = \argmin_{x}\left\{\Psi^*(x) -\langle x,y  \rangle\right\}.
    \end{align}
    By Fenchel-Young, $\Psi^*(x) -\langle x,y\rangle \geq -\Psi(y)$, with exact equality if and only if $x \in \partial \Psi(y).$ We can therefore conclude that $\prox_\varphi(y)= \partial \Psi(y)$. Next, evaluating the supremand defining $\Psi^*$~\cref{eq:convex_conjugate} at $x/\alpha$ gives 
    \begin{align}
        \Psi^*(x) \geq \langle x,x/\alpha\rangle -\Psi(x/\alpha) &\geq \frac{1}{\alpha}\|x\|^2 - \frac{\alpha}{2}\left\| \frac{x}{\alpha}\right\|^2 - L\left\|\frac{x}{\alpha}\right\|-K
        \\&=\frac{1}{2\alpha}\|x\|^2 -\frac{L}{\alpha}\|x\|-K,
    \end{align}
    by the upper bound assumed on $\Psi$. Therefore, $\varphi(x) \geq \tfrac{1-\alpha}{2\alpha}\|x\|^2-\tfrac{L}{\alpha}\|x\|-K$, an upward facing quadratic, since $\tfrac{1-\alpha}{2\alpha}>0$. Coercivity of $\varphi$ follows. 
\end{proof}

It is straightforward to constrain a real-valued neural network $\Psi_\theta$ to be convex, as we shall detail in~\cref{section:LPNs}. If we take $y\in \partial \Psi(x)$, then $y \in \prox_\varphi(x)$ for some $\varphi$. Immediately, from~\cref{def:proximal_operator}, 
\begin{equation}\label{eq:prox_decrease}
     \varphi(y) + \frac{1}{2}\|x-y\|^2 \leq  \varphi(x).
\end{equation} 
Thus, descent is built directly into the architecture. If application of the proximal operator is preceded by a gradient step on some $L_f$-smooth term $f$, the resulting update retains a corresponding descent property. Indeed, suppose we take $x_{k+1} \in \prox_{\varphi}(x_k-\eta \nabla f(x_k))$ where $f$ is $L_f$-smooth and $\eta < 1/L_f$. Then by  \cref{eq:prox_decrease} with $y=x_{k+1}$, $x = x_k-\eta \nabla f(x_k)$, and the descent lemma~\cite[Lemma 5.7]{beck2017first} we have
\begin{equation}\label{eq:generic_decrease}
    F(x_k) - F(x_{k+1}) \geq \frac{1-\eta L_f}{2}\|x_{k+1}-x_k\|^2, \quad F \coloneqq \eta f + \varphi.
\end{equation}
Since the ICNNs we consider (\cref{def:ICNN}) satisfy the quadratic-growth condition in~\cref{lem:coercive_regularizer}, $\varphi$ is coercive and hence bounded below. Convergence of $\{F(x_k)\}$ then follows since $f$ is typically assumed to be bounded below. Summing~\cref{eq:generic_decrease} gives $\sum_k\|x_{k+1}-x_k\|^2 < \infty$, and the KL property upgrades this to convergence of the iterates themselves, outlined in~\cref{prop:generic_convergence_kl}. This argument is the template for the convergence results in this paper. Critically, no step will depend on how the network was trained: the guarantee follows from the ICNN architecture alone.

Finally, we show that when the convexity of $\Psi$ is upgraded to strong convexity, $\varphi$ in~\cref{eq:bar_phi} is differentiable. This allows us to analyze objectives involving $\varphi$ using ordinary gradients, rather than subgradients, substantially simplifying the arguments that follow. Recall that a function is $L$-smooth if it has an $L$-Lipschitz gradient.
\begin{lemma}\label{lem:C_1_regularizer}
    Consider $\Psi$ and $\varphi$ as in~\cref{lem:coercive_regularizer}. If $\Psi$ is $\alpha$-strongly convex for some $\alpha>0$, then $\varphi$ is $(1+1/\alpha)$-smooth. 
\end{lemma}
\begin{proof}
By~\cite[Theorem 3]{kakade2012regularization}, $\Psi^*$ is $1/\alpha$-smooth, so $\varphi$ is $(1+1/\alpha)$-smooth by~\cref{eq:bar_phi}.
\end{proof}

We now turn to the Kurdyka--\L ojasiewicz (KL) property, the machinery needed to upgrade vanishing gradients and step sizes into convergence of the iterates themselves. For the objective functions $F$ of interest, we shall see that it is straightforward to show that $\|\nabla F(x_k)\|$ and $\|x_k-x_{k+1}\|$ go to zero along the sequence $\{x_k\}_{k\geq 1}$ generated by some PnP algorithm. However, neither condition implies iterate convergence to a critical point of $F$. The objective may become very flat around a critical point, so that the gradient vanishes while the iterates travel an infinite total distance without converging. For a simple example, take $x_0 = 0$ and $x_{k} = x_{k-1} + \frac{1}{k}$ for $k \geq 1$. Then $\|x_k - x_{k+1}\| = \frac{1}{k+1} \to 0$, and $\|\nabla F(x_k)\|\to 0$ for many choices of $F$, but the iterates clearly diverge. Further, $\sum_k \|x_k - x_{k+1}\|^2 < \infty$, so square summability is insufficient. 

The KL property rules out precisely this kind of degenerate behavior. Informally, it says the objective gap can be monotonically reparametrized so that its gradient remains bounded away from zero, except \emph{at} a critical point. Thus, while $\|\nabla F\|$ may be small near a critical point, it cannot vanish arbitrarily. To make this precise, define the reparametrized objective gap to be $\Phi(x) \coloneqq \varphi(|F(x)-F(x^*)|)$ for an appropriate $\varphi$ and a critical point $x^*$. If $F$ is differentiable and $\varphi'>0$, then the gradient magnitude of the reparametrized objective gap becomes\begin{equation}\label{eq:KL_reparam}
    \|\nabla \Phi(x)\| = \|\varphi'(|F(x)-F(x^*)|)\nabla F(x)\| = \varphi'(|F(x)-F(x^*)|)\|\nabla F(x)\| .
\end{equation} The KL property requires this quantity to be bounded away from zero whenever $F(x) > F(x^*)$. We restrict our attention to objectives that are at least $C^1$, though the property can be defined for a wider class of functions.

\begin{definition}[The KL Property]\label{def:KL}
Let $F:\R^n\to\R$ be continuously differentiable, and let $x^*\in\R^n$.
We say that $F$ has the KL property at $x^*$ if
there exist $\eta\in(0,+\infty]$, a neighborhood $U$ of $x^*$, and a continuous concave\footnote{If $\varphi$ is concave, then $\varphi(s) - \varphi(t) \geq \varphi'(s)(s-t)$ holds for all $0 \leq t \leq s$. This allows decreases in the objective value to be easily related to the lengths of successive steps.} function $\varphi:[0,\eta)\to\R_+$ such that:
\begin{enumerate}[label=\normalfont(\roman*)]
    \item \label{def:KL_i} $\varphi(0)=0$;
    \item \label{def:KL_ii} $\varphi\in C^1((0,\eta))$;
    \item \label{def:KL_iii} $\varphi'(s)>0$ for every $s\in(0,\eta)$;
    \item \label{def:KL_iv} for every $x\in U$ satisfying $F(x^*)<F(x)<F(x^*)+\eta$, the KL inequality holds: \begin{equation}\tag{KL}\label{eq:kl_inequality}
        \varphi'\bigl(F(x)-F(x^*)\bigr)\|\nabla F(x)\|
        \geq 1.
    \end{equation}
\end{enumerate}
The function $\varphi$ is called a \emph{desingularizing function}. A continuously differentiable function satisfying the KL property at every point of $\R^n$ is called \emph{KL}, or a \emph{KL function}.
\end{definition}

One common route to establishing the KL property is to show that the function's graph belongs to a sufficiently ``tame'' geometric class. Informally, these classes exclude sufficiently pathological behavior so that a suitable reparameterization $\varphi$ exists. These tame classes contain the functions used in practice, including those built from polynomials, exponentials, logarithms, and piecewise algebraic functions through standard operations like sums, products and compositions. We formalize one such notion of geometric tameness known as o-minimal structures in~\cref{section:KL_objective}. The KL property can also be established directly, without appealing to geometric tameness. In particular, every strongly convex $C^1$ function is KL. Indeed, a $\mu$-strongly convex $C^1$ function $F$ satisfies\begin{equation}
    F(x) - F(x^*) \leq \langle \nabla F(x),x-x^*\rangle - \frac{\mu}{2}\|x-x^*\|^2 \leq \frac{1}{2\mu}\|\nabla F(x)\|^2,
\end{equation}
for any $x,x^*\in\R^n$. Therefore, $\|\nabla F(x)\| \geq \sqrt{2\mu\bigl(F(x)-F(x^*)\bigr)}$ whenever $F(x)>F(x^*)$,  which is exactly \cref{eq:kl_inequality} for the choice of $\varphi(s) = \sqrt{\tfrac{2s}{\mu}}$.

\section{Learned Proximal Networks (LPNs)}\label{section:LPNs}
In this section, we outline the procedure for constructing strongly convex potentials that satisfy the Lipschitz condition of~\cref{lem:coercive_regularizer} using Input Convex Neural Networks (ICNNs)~\cite{amos2017input}.

\begin{definition}[Input Convex Neural Network (ICNN)]
\label{def:ICNN}
Let $\psi_\theta:\R^n\to \R$ be a $K$-layered neural network given by $\psi_\theta(x) \coloneqq z_K$, with hidden states 
\begin{align}
    z_{i+1}\coloneqq \sigma(W_iz_i + H_i x + b_i), \; i =0,\dots,K-1, \quad W_0,z_0 \equiv0.
\end{align}
$H_i$ and $W_i$ are linear operations with $W_i$ having non-negative entries. $b_i$ are biases, and $\sigma:\R \to \R$ is a convex and non-decreasing activation, applied element-wise. We call $\psi_\theta$ an ICNN with parameters $\theta = \{W_i,H_i,b_i\}_{i=0}^{K-1}$.
\end{definition}

The architecture of an ICNN preserves convexity layer by layer. At the $i$th layer, suppose each component $x\mapsto[z_i(x)]_j$ is convex. Since $W_i$ has nonnegative entries, each $[W_iz_i(x)]_k$ is a nonnegative linear combination of convex functions, and is therefore convex. Adding the affine term $H_ix + b_i$ preserves convexity. Finally, because the activation $\sigma$ is convex and nondecreasing, applying $\sigma$ coordinatewise preserves convexity, and we can conclude that $\psi_\theta$ is convex with respect to the input. We now define the class of \emph{Learned Proximal Networks} (LPNs) considered throughout this work. 
\begin{definition}[Learned Proximal Network]\label{def:LPN}
$\lpn: \R^n \to \R^n$ is called an LPN with \emph{strong convexity modulus} $\alpha$ if it can be written as 
\begin{align}
    \lpn(y)  \in \partial \psi_{\theta}(y) + \alpha y,
\end{align}
for some $0<\alpha<1$, where $\psi_\theta$ is ICNN as in~\Cref{def:ICNN}, with a Lipschitz activation function $\sigma$. We call $\potential \coloneqq \psi_{\theta} + \tfrac{\alpha}{2}\|\cdot\|^2$ the LPN's \emph{potential function} and $\reg \coloneqq$ $\potential^* - \tfrac{1}{2}\|\cdot\|^2$ its \emph{implicit regularizer}.
\end{definition}

\begin{remark}[Regularity of the implicit LPN regularizer]\label{remark:regularity_of_phi}
By~\cref{lem:coercive_regularizer}, any LPN $f_{\theta,\alpha}$ satisfies
$f_{\theta,\alpha}(x)\in\prox_{\reg}(x)$ globally. Moreover, its potential $\potential$ satisfies the growth condition of~\cref{lem:coercive_regularizer}, and hence the associated regularizer $\reg$ is coercive. Indeed, if $\sigma$ is $L_\sigma$-Lipschitz, $\sigma$ acting coordinatewise is also $L_\sigma$-Lipschitz. Suppose the hidden state $z_i$ is $L_i$-Lipschitz, then it holds that \begin{equation}
    \|z_{i+1}(x)-z_{i+1}(y)\| \leq L_\sigma (\|W_i\| L_i + \|H_i\|)\|x-y\|,
\end{equation}
for all $x,y\in \R^n$, so that $\psi_{\theta}$ is globally $L$-Lipschitz for some $L$ by induction. Therefore, $\psi_\theta(x) \leq   L\|x\| + \psi_\theta(0)$, meaning $\potential(x) \leq \tfrac{\alpha}{2}\|x\|^2 + L\|x\| + \psi_\theta(0)$ for all $x$. Next, $\potential$ is clearly $\alpha$-strongly convex, so $\reg$ is $C^1$ by~\cref{lem:C_1_regularizer}. 
\end{remark}

In terms of implementing LPNs, when the activation is differentiable, the subdifferential $\partial \psi_\theta(y)$ reduces to the singleton $\{\nabla \psi_\theta(y)\}$, and evaluation of the LPN amounts to differentiation. Even when the activations are not differentiable everywhere, evaluation of an LPN is straightforward. For a given input $y$, if any preactivation lies in a region where the activation is non-differentiable, one simply chooses \emph{any} subgradient and propagates this choice through the network to obtain an element of $\partial \psi_\theta(y)$. The entire subdifferential need not be computed. For example, the ReLU activation $\mathrm{ReLU}(t)=\max\{0,t\}$, is not differentiable at $t=0$, but can be an activation in an LPN. At a zero preactivation, one may select any value in the subdifferential $\partial \mathrm{ReLU}(0) = [0,1]$, such as $0$. This procedure is analogous to the standard backward-pass convention used for training neural networks with ReLU activations. In particular, PyTorch\cite{paszke2019pytorch} selects the minimum-norm element of the subdifferential when the derivative of a convex activation does not exist.\footnote{See \url{https://docs.pytorch.org/docs/2.13/notes/autograd}.}

In the following section, we make use of LPNs to construct convergent PnP algorithms. In particular, these algorithms minimize a regularized objective of the form
\begin{equation}\label{eq:lpn_regularized_objective}
    F(x) \coloneqq \eta \fid(x) + \reg(x),
\end{equation} where $\reg$ is an LPN's implicit regularizer, and $\fid$ is an $L_{\mathbf{fid}}$-smooth data-fidelity term, with $0 < \eta < 1/L_{\mathbf{fid}}$. The proximal structure of the LPN yields the standard descent and first-order convergence guarantees of proximal methods. Under the additional assumption that $F$ is KL, these results can be made stronger. 
    In practice, the KL requirement is mild. Whenever the LPN activation function and the data-fidelity term $\fid$ in question are constructed from polynomials, exponentials, logarithms, absolute values, maxima and minima, or restricted analytic functions through finitely many sums, products, and compositions, the resulting regularized objective $F \coloneqq \eta \fid + \reg$ is KL. Examples of such activations are ReLU, Huberized ReLU (see \cref{eq:huberized_relu} for the definition), or Softplus $S_\beta(x) \coloneqq \tfrac{1}{\beta}\log(1+\exp(\beta x))$. We make this precise in the following result, which we prove in the supplementary materials,~\cref{section:KL_objective}.
\begin{proposition}[KL property of LPN objectives]\label{prop:lpn_objective_KL}
    Let $\lpn$ be an LPN as in~\cref{def:LPN}, with potential $\potential$ 
    and implicit regularizer $\reg$. Let $\fid:\R^n\to\R$ be continuously differentiable. Suppose that the activation function $\sigma$ used to construct $\lpn$ and the data-fidelity term $\fid$ are constructed from polynomials, exponentials, logarithms, absolute values, maxima and minima, or restricted analytic functions through finitely many sums, products, and compositions. Then, for every $\eta>0$ and $\gamma>0$, the objective $F(x)\coloneqq \eta \fid(x)+\gamma\reg(x)$ is KL. Moreover, the augmented objective $\xi:\R^{n+1}\to\R$ given by $\xi(x,t)\coloneqq F(x)+\frac{1}{2}t^2$, and the averaged objective, given by $F^{(\lambda)} \coloneqq \eta \fid + \reg^{(\lambda)}$, for $\reg^{(\lambda)} \coloneqq \lambda M_{1-\lambda} \reg$ are KL.
\end{proposition}

\section{Convergence Analysis of LPN Algorithms}\label{section:convergence_analysis}

We first present the machinery that we shall use to show convergence of LPN-based PnP methods. We focus on the PnP variant of proximal gradient descent~\cref{eq:PnP-PGD}, though similar results hold for the alternating direction method of multipliers (ADMM) and the proximal point method (PPM), which we present in~\cref{section:other_methods}. In each application, given an LPN $\lpn$, we consider the continuously differentiable objective $F$ in~\cref{eq:lpn_regularized_objective}. We show that the corresponding algorithm produces iterates $\{x_k\}_{k \geq 1}$ satisfying the following \textbf{standard conditions}:
\begin{hypotheses}
\item \label[hypothesis]{sufficient_decrease}
(Sufficient Decrease): For some $a>0$ and all $k\geq 1$,
$F(x_{k+1}) + a\|x_k-x_{k+1}\|^2 \leq F(x_k)$.
\item \label[hypothesis]{relative_error}
(Relative Error): For some $b>0$ and all $k\geq 1$,
$\|\nabla F(x_{k+1})\| \leq b\|x_k-x_{k+1}\|$.
\item \label[hypothesis]{precompactness}
(Objective-Precompactness): There exists a subsequence
$\{x_{k_j}\}_{j\geq1}$, a point $x^*\in\R^n$, and some $F^*\in\R$
such that $x_{k_j}\to x^*$ and $F(x_{k_j})\to F^*$.
\end{hypotheses}
We next show that these standard conditions imply monotone objective descent and square-summability of the steps.
 \begin{proposition}\label{prop:generic_convergence}
     Suppose $F:\R^n\to {\R}$ and the sequence $\{x_k\}_{k \geq 0}$ satisfy \cref{sufficient_decrease}, \cref{relative_error}, and \cref{precompactness}. Then the following hold:
     \begin{enumerate}[label=\normalfont(\roman*)]
    \item \label{prop:generic_convergence_i} $F(x_k)$ is non-increasing and convergent.
    \item \label{prop:generic_convergence_ii} $\sum_{k=1}^{+\infty} \|x_k-x_{k+1}\|^2 < +\infty$.
    \end{enumerate}
 \end{proposition}
 \begin{proof}
 \cref{sufficient_decrease} and ~\cref{precompactness} imply the existence of some $F^* \in \R$ where $F(x_{k_j}) \downarrow F^*$, from which~\ref{prop:generic_convergence_i} follows.  Summing~\cref{sufficient_decrease} from $k=1$ to $K$, 
\begin{equation}
\sum_{k=1}^K\|x_k - x_{k+1}\|^2\leq \frac{F(x_1)-F(x_{K+1})}{a} \leq \frac{F(x_1)-F^*}{a},
\end{equation}
and~\ref{prop:generic_convergence_ii} follows when we let $K \to \infty$. 
 \end{proof}

The additional assumption that the objective is KL upgrades these guarantees to iterate convergence and the finite-length property, which we present next.  
\begin{proposition}(\cite[Theorem 2.9]{attouch2013convergence})\label{prop:generic_convergence_kl}
Suppose $F:\R^n\to {\R}$ and the sequence $\{x_k\}_{k \geq 0}$ satisfy~\cref{sufficient_decrease},~\cref{relative_error}, and~\cref{precompactness}. If $F$ is a KL function, then the following hold: 
\begin{enumerate}[
    label=\normalfont(\roman*),
    ref=\thetheorem(\roman*)]
\crefalias{enumi}{proposition}
\item $\{x_k\}_{k\geq 1}$ converges.
\item $\sum_{k=1}^{+\infty} \|x_k-x_{k+1}\| < +\infty$.
\end{enumerate}
\end{proposition}
\begin{remark}[On convergence rates]\label{remark:on_convergence_rates}
$\mathcal{O}(K^{-1/2})$ and $\mathcal{O}(K^{-1})$ rates follow immediately from square summability and the finite-length property, respectively. Indeed, if the sequence $a_k = \|x_{k+1}-x_k\|$ satisfies $\sum_{k=1}^{\infty} a_k^2= C_1$ for some $C_1 < \infty$, then
\begin{equation}
    \min_{1\leq k\leq K} a_k^2
    \leq
    \frac{1}{K}\sum_{k=1}^K a_k^2
    \leq
    \frac{1}{K}\sum_{k=1}^{\infty} a_k^2= \frac{C_1}{K},
\end{equation}
and therefore $\min_{1\leq k\leq K} \|x_k - x_{k+1}\|
    =
    \mathcal{O}(K^{-1/2})$ after taking a square-root. 
Similarly, if a nonnegative sequence $\{a_k\}_{k\geq 1}$ satisfies $\sum_{k=1}^{\infty} a_k = C_2$ for some $C_2 < \infty$, then
\begin{equation}
    \min_{1\leq k\leq K} a_k
    \leq
    \frac{1}{K}\sum_{k=1}^K a_k
    \leq
    \frac{1}{K}\sum_{k=1}^{\infty} a_k
    =
    \frac{C_2}{K},
\end{equation}
and therefore $
    \min_{1\leq k\leq K} \|x_k - x_{k+1}\|
    =
    \mathcal{O}(K^{-1})$. Moreover, the same rates hold for the gradient of the objective. The optimality condition for $x_{k+1}=\prox_\varphi(x_k-\eta\nabla f(x_k))$ reads $\nabla\varphi(x_{k+1}) + x_{k+1}-x_k+\eta\nabla f(x_k)=0$, and therefore if $f$ is $L_f$-smooth,
    \begin{equation}
        \|\nabla F(x_{k+1})\|
        = \big\|\eta\big(\nabla f(x_{k+1})-\nabla f(x_k)\big) - (x_{k+1}-x_k)\big\|
        \leq (1+\eta L_f)\,\|x_{k+1}-x_k\|.
    \end{equation}
Consequently, $\min_{1\le k\le K}\|\nabla F(x_{k+1})\|$ inherits the $\mathcal{O}(K^{-1/2})$ and $\mathcal{O}(K^{-1})$ rates above. Since these implications are standard, we state results in the remainder of this work only in terms of square summability or the finite-length property, with the corresponding rates implied.
\end{remark}
\subsection{Standard LPN-Proximal Gradient Descent}

We next consider standard PnP-PGD using an LPN. It is standard to assume that the data-fidelity term $\fid$ is $L_{\mathbf{fid}}$-smooth, and that the step size $\eta$ satisfies $\eta < 1/L_{\mathbf{fid}}$. This is a very mild assumption, because in most inverse problems, the Lipschitz constant of $\nabla \fid$ can be computed accurately and used to select an appropriate step size. For example, in the case of the least-squares data-fidelity term $\fid(x)\coloneqq \tfrac{1}{2}\|Ax-y\|^2$, $L_{\mathbf{fid}} = \|A\|^2 = \lambda_{\mathrm{max}}(A^TA),$ which can be reliably calculated via power-iterations. The resulting algorithm is given in~\cref{alg:lpn_pgd_vanilla}.

\begin{algorithm}[H]
\caption{LPN-PGD}
\label{alg:lpn_pgd_vanilla}
\begin{algorithmic}[1]
\REQUIRE Initial point $x_0\in \R^n$, $L_{\mathbf{fid}}$-smooth and bounded below $\fid:\R^n \to \R$,
         step size $0<\eta<1/L_{\mathbf{fid}}$, LPN $\lpn$ as in~\cref{def:LPN},  number of steps $K$.
\FOR{$k=0,1,\dots,K-1$}
    \STATE
    $x_{k+1} \coloneqq \lpn\!\left(x_k-\eta\nabla\fid(x_k)\right)$
\ENDFOR
\RETURN $x_K$
\end{algorithmic}
\end{algorithm}

\begin{proposition}[Convergence of LPN-PGD]\label{prop:pgd_convergence}
Consider $\{x_k\}_{k \geq 0}$, the iterates of standard LPN-PGD in~\cref{alg:lpn_pgd_vanilla}, where $x_0$ is the initial point, $\reg$ denotes the regularizer associated with $\lpn$, and $F\coloneqq \eta \fid + \reg$. Then the following hold:
\begin{enumerate}[label=\normalfont(\roman*)]
\crefalias{enumi}{proposition}
    \item \label{prop:pgd_convergence_i} $F(x_k)$ is non-increasing and convergent.
    \item \label{prop:pgd_convergence_ii} $\sum_{k=1}^{+\infty} \|x_k-x_{k+1}\|^2 < +\infty$.
    \item \label{prop:pgd_convergence_iii} If $F$ is KL, then $\sum_{k=1}^{+\infty}\|x_k-x_{k+1}\| < +\infty,$ and $x_k \to x^*$, with $x^*$ a critical point of $F$.
    \end{enumerate}
\end{proposition} 

\begin{proof}
    First, by~\cref{lem:coercive_regularizer} and~\cref{lem:C_1_regularizer}, $\lpn (y) \in \prox_{\reg}(y)$ for all $y \in \R^n$, where $\reg$ is coercive and $C^1$. Therefore, we have $x_{k+1} \in \prox_{\reg}(x_k - \eta \nabla \fid(x_k))$ for all $k$. We now show that the standard conditions apply to $\{x_k\}$ and $F$.
    \paragraph{\cref{sufficient_decrease}} $x_{k+1}$ minimizes $x \mapsto \reg(x) + \tfrac{1}{2}\|x-(x_k-\eta \nabla \fid(x_k))\|^2,$ meaning
    \begin{equation}
        \reg(x_{k+1})+\frac{1}{2}\|x_{k+1}-x_k +\eta \nabla \fid(x_k)\|^2 \leq \reg(x_k)+\frac{1}{2}\|\eta \nabla \fid (x_k)\|^2.
    \end{equation}
    Expanding the square, we can write this as
    \begin{equation}\label{eq:pgd_convergence1}
        \reg(x_{k+1})\leq \reg(x_k)-\langle \eta \nabla \fid(x_k),x_{k+1}-x_k\rangle -\frac{1}{2}\|x_k - x_{k+1}\|^2.
    \end{equation}
    By the descent lemma; see~\cite[Lemma 5.7]{beck2017first}, for example, \begin{equation} \fid(x_{k+1}) \leq \fid(x_k) + \langle \nabla \fid(x_k),x_{k+1}-x_k\rangle + \frac{L_{\mathbf{fid}}}{2}\|x_k - x_{k+1}\|^2.\end{equation} Multiplying this inequality by $\eta$ and combining with~\cref{eq:pgd_convergence1},\begin{equation}
        F(x_{k+1}) + \frac{1-\eta L_{\mathbf{fid}}}{2}\|x_k - x_{k+1}\|^2\leq F(x_k),
    \end{equation}
    which is exactly~\cref{sufficient_decrease} with $a = \tfrac{1-\eta L_{\mathbf{fid}}}{2} > 0$, a positive number.   
    \paragraph{\cref{relative_error}}
    
    By the first-order optimality of $x_{k+1}$ in the proximal subproblem, \begin{equation}
        \nabla \reg(x_{k+1}) + (x_{k+1} - (x_k - \eta \nabla \fid (x_k))) = 0. 
    \end{equation}
    Adding $\eta \nabla \fid(x_{k+1})$ to each side and rearranging,
    \begin{equation}
         \nabla \reg(x_{k+1}) + \eta \nabla \fid(x_{k+1}) = x_k - x_{k+1} +\eta \nabla \fid (x_k)- \eta \nabla \fid(x_{k+1})
    \end{equation}
    and taking a norm, we get
        $\|\nabla F(x_{k+1})\| \leq (1+\eta L_{\mathbf{fid}})\|x_k -x_{k+1}\| \leq 2\|x_k - x_{k+1}\|,$
    which is~\cref{relative_error}.

    \paragraph{\cref{precompactness}}  Since $F(x_k) \leq F(x_0)$, all iterates lie in the sublevel set $\{F \leq F(x_0)\}$. This set is compact since $\fid$ is bounded below and $\reg$ is coercive, making $F$ coercive ($F$ is continuous, hence lsc). The sequence is therefore bounded, and we can find a convergent subsequence, and~\cref{precompactness} follows, again by the continuity of $F$.

    \ref{prop:pgd_convergence_i} and \ref{prop:pgd_convergence_ii} therefore follow from~\cref{prop:generic_convergence}, and \ref{prop:pgd_convergence_iii} follows from~\cref{prop:generic_convergence_kl}.
\end{proof}

\subsection{Denoiser Averaging}\label{subsection:denoiser_averaging}
When a denoiser enforces excessive regularization in a PnP algorithm, usually indicated by oversmoothing or artifacts, it is common practice to average it with the identity~\cite{goujon2023neural,tan2023provably,martin2025pnp,bohra2021learning}. Given a denoiser $D_\sigma$ and an averaging parameter $\lambda \in (0,1)$, the averaged denoiser is given by \begin{equation}\label{eq:averaged_denoiser} 
    D_\sigma^{(\lambda)} \coloneqq (1-\lambda)I_n   + \lambda D_\sigma.
\end{equation}
For a generic denoiser, averaging need not preserve its original interpretation or even model class. One advantage of the LPN framework is that the average of an LPN with the identity is still an LPN, which we make precise in the following result. 
\begin{proposition}[Closure of LPNs under averaging]\label{prop:closure_of_lpns}
    Let $\lpn$ be an LPN with potential $\potential$ and implicit regularizer $\reg$. For any $\lambda \in (0,1)$, define \begin{equation}\label{eq:lpn_averaged}
        \lpn^{(\lambda)} \coloneqq (1-\lambda)I_n + \lambda \lpn.
    \end{equation}
    Then $\lpn^{(\lambda)}$ is an LPN with strong convexity modulus $\alpha_\lambda \coloneqq 1-\lambda(1-\alpha)$.
\end{proposition}
\begin{proof}
    By construction, $\lpn(x) \in \partial \psi_\theta(x) + \alpha x$, so 
    \begin{align}\nonumber
        \lpn^{(\lambda)}(x) & = (1-\lambda)x+\lambda \lpn(x) \in \lambda \partial \psi_\theta(x) + (1-\lambda + \lambda \alpha)x \\
        & = \partial (\lambda \psi_\theta)(x) + \alpha_\lambda x = \partial \potential ^{(\lambda)}(x).
    \end{align}
    $\lpn^{(\lambda)}$ is then clearly an LPN with strongly convex potential $ \potential ^{(\lambda)} \coloneqq \lambda \potential + \frac{1-\lambda}{2}\|\cdot\|^2$. The rest of the result follows from the fact that $\lambda \psi_\theta$ is also an ICNN and $\alpha_\lambda \in (0,1)$.
\end{proof}

The regularizer associated with the averaged LPN also admits a simple variational interpretation which we make precise in the following corollary. 

\begin{corollary}
\label{cor:averaged_lpn}
Let $\lpn$ be an LPN with implicit regularizer $\reg$, let $\lambda\in(0,1)$, and define $\lpn^{(\lambda)}$ as in~\cref{eq:lpn_averaged}. Consider $\{x_k\}_{k \geq 1}$, the iterates of standard LPN-PGD in~\cref{alg:lpn_pgd_vanilla} using $\lpn^{(\lambda)}$ with implicit regularizer $\reg^{(\lambda)}$, and let $F^{(\lambda)}\coloneqq \eta\fid + \lambda M_{1-\lambda}\reg$. Then $\reg^{(\lambda)} = \lambda M_{1-\lambda}\reg$ and the following hold:
\begin{enumerate}[label=\normalfont(\roman*)]
\crefalias{enumi}{proposition}
    \item$F^{(\lambda)}(x_k)$ is non-increasing and convergent.
    \item $\sum_{k=1}^{+\infty} \|x_k-x_{k+1}\|^2 < +\infty$.
    \item If in addition $F^{(\lambda)}$ is KL, then $\sum_{k=1}^{+\infty}\|x_k-x_{k+1}\| < +\infty,$ and $\{x_k\}_{k \geq 1}$ converges to a critical point $x^*$ of $F^{(\lambda)}$. 
    \end{enumerate}
\end{corollary}
\begin{proof}
In the following, we use the infimal convolution identity [Theorem 4.17]\cite{beck2017first}. From \cref{prop:closure_of_lpns}, $\lpn^{(\lambda)}$ is an LPN with potential $\potential ^{(\lambda)}(x)$, so we can write its implicit regularizer as 
\begin{align}
\reg^{(\lambda)}(x)
&=\left(\lambda \potential+\frac{1-\lambda}{2}\|\cdot\|^2\right)^*(x)-\frac{1}{2}\|x\|^2 \nonumber\\
&=\inf_{z\in\R^n}\left\{\lambda \potential^*(z)+\frac{1}{2(1-\lambda)}\|x-\lambda z\|^2\right\}-\frac{1}{2}\|x\|^2,
\end{align}
after a change of variables $z\to \lambda z$. Next, since $\potential^*(z)=\reg(z)+\frac{1}{2}\|z\|^2$,
\begin{align}
\reg^{(\lambda)}(x)
&=\inf_{z\in\R^n}\left\{\lambda\reg(z)+\frac{\lambda}{2}\|z\|^2+\frac{1}{2(1-\lambda)}\|x-\lambda z\|^2-\frac{1}{2}\|x\|^2\right\} \nonumber\\
&=\lambda\inf_{z\in\R^n}\left\{\reg(z)+\frac{1}{2(1-\lambda)}\|x-z\|^2\right\}
=\lambda M_{1-\lambda}\reg(x),
\end{align}
Since $\tfrac{\lambda}{2}\|z\|^2 +\tfrac{1}{2(1-\lambda)}\|x-\lambda z\|^2 -\tfrac{1}{2}\|x\|^2 = \tfrac{\lambda}{2(1-\lambda)}\|x-z\|^2.$ We can conclude that $\reg^{(\lambda)} = \lambda M_{1-\lambda}\reg$, and the result follows from \cref{prop:pgd_convergence}.
\end{proof}

Thus, averaging the denoiser with the identity corresponds exactly to replacing the learned regularizer by a scaled Moreau-smoothed version of itself. The averaging parameter $\lambda$ therefore provides a principled way to weaken the learned prior, while preserving the LPN structure and the associated convergence guarantee. We note that the objective $F^{(\lambda)}$ associated with this averaged denoiser is KL whenever the LPN objective and data fidelity term are constructed from reasonably standard functions. This is made precise in~\cref{prop:lpn_objective_KL}.

\subsection{Regularizer Scaling via Inexact PGD}\label{subsection:regularizer_scaling} In this subsection, we present a method to minimize objectives of the form $F(x)=\eta \fid(x)+\gamma \reg(x)$ for $0<\gamma\leq 1$, through inexact proximal gradient descent~\cite{lauga2026characterizations,schmidt2011convergence,gu2018inexact,yao2016efficient,bonettini2020convergence}
Whereas denoiser averaging weakens the LPN by replacing $\reg$ with a Moreau-smoothed surrogate, one may instead wish to change its strength by simply scaling the regularizer. For simplicity, we assume that the LPN activation is differentiable. The same construction can be formulated for nondifferentiable activations using subgradient optimality conditions. Given some LPN $\lpn$ with potential function $\potential$ and implicit regularizer $\reg$, we seek to evaluate 
$\prox_{\gamma \reg}$ for some $\gamma > 0$. Define $w^*$ as
 \begin{equation}\label{eq:w_star}
    w^* \coloneqq \argmin_w G(w), \qquad G(w) \coloneqq \frac{\gamma}{2}\|w\|^2 + (1-\gamma)\potential(w) - \langle x,w\rangle,
\end{equation}
then $\prox_{\gamma \reg}(x) = \lpn (w^*).$ Indeed, since $\potential = \psi_{\theta} + \tfrac{\alpha}{2}\|\cdot\|^2$ for some $0<\alpha<1$,
\begin{equation}
    G(w) = (1-\gamma)\psi_\theta(w) + \frac{\gamma + (1-\gamma)\alpha}{2}\|w\|^2 - \langle x,w\rangle. 
\end{equation}
For $0 < \gamma \leq 1$, $(1-\gamma)\psi_\theta$ is convex, so $G$ is $(\gamma+(1-\gamma)\alpha)$-strongly convex, and therefore admits a unique minimizer $w^*$. Setting $z^* \coloneqq \lpn(w^*) = \nabla \potential (w^*)$, $w^* = \nabla \potential^*(z^*)$ from~\cref{eq:conjugate_gradient_identity}. Since $\nabla \reg(z) = \nabla\potential^*(z)-z$, \begin{equation}
    \gamma \nabla \reg(z^*)+z^*-x = \gamma(w^*-z^*) + z^*-x 
    = \gamma w^*+(1-\gamma)z^*-x = 0,
\end{equation}
where the last equality follows from the first-order optimality condition of~\cref{eq:w_star}. The result is exactly the first-order optimality condition for $\prox_{\gamma \reg}$. Indeed, $\reg$ is $1$-weakly convex by~\cref{eq:bar_phi}, so $z\mapsto \gamma \reg(z) + \tfrac{1}{2}\|z-x\|^2$ is convex for $0<\gamma \leq 1$. We can therefore conclude that $\lpn (w^*) =  \prox_{\gamma \reg}(x)$.

When $\potential$ is differentiable, $G$ can be minimized using standard first-order methods, since its gradient can be evaluated directly from the original LPN:
    $\nabla G(w) = \gamma w + (1-\gamma)\lpn (w)-x.$
In terms of implementation, one could solve~\cref{eq:w_star} to convergence at each iterate, thereby recovering the exact scaled proximal step $\prox_{\gamma \reg}(x_k - \eta \nabla \fid(x_k))$. In that case, the convergence results developed above apply directly. However, solving the inner problem to such high accuracy is unnecessary. Because $G$ is strongly convex and its optimality residual is directly computable, one can terminate the inner solve early and still retain convergence guarantees. In particular, it is sufficient to require the residual to decay at a prescribed summable rate, leading to the inexact scaled LPN-PGD scheme in~\cref{alg:lpn_pgd_scaled}, whose convergence guarantees are established in the following result.

\begin{algorithm}[H]
\caption{Inexact Scaled LPN-PGD}
\label{alg:lpn_pgd_scaled}
\begin{algorithmic}[1]
\REQUIRE Initial point $x_0\in\R^n$, $L_{\mathbf{fid}}$-smooth and bounded below $\fid:\R^n\to\R$, step size $0<\eta<1/L_{\mathbf{fid}}$, LPN $\lpn$ as in~\cref{def:LPN} with potential $\potential$, regularizer scaling $0<\gamma\leq1$, number of steps $K$, and a fixed constant $C>0$. Further, the activation of $\lpn$ must be differentiable.
\FOR{$k=0,1,\dots,K-1$}
    \STATE Let \[G_k(w)\coloneqq \frac{\gamma}{2}\|w\|^2+(1-\gamma)\potential(w)-\langle x_k-\eta\nabla\fid(x_k),w\rangle\]
    
    \STATE Find $w_{k+1}$ such that
    \[
    \|\nabla G_k(w_{k+1})\|\leq \frac{C}{(k+1)^{p}}, \quad {p}>3/2.
    \]
    \STATE $x_{k+1}\coloneqq \lpn(w_{k+1})$
\ENDFOR
\RETURN $x_K$
\end{algorithmic}
\end{algorithm}

\begin{proposition}[Convergence of Inexact Scaled LPN-PGD]\label{prop:scaled_pgd_convergence}
Consider the iterates of inexact scaled LPN-PGD in~\cref{alg:lpn_pgd_scaled}, $\{x_k\}_{k\geq 1}$, where $\reg$ denotes the implicit regularizer associated with $\lpn$, and define
\begin{equation}
    F \coloneqq \eta \fid + \gamma \reg.
\end{equation}
Then the following hold:
\begin{enumerate}[label=\normalfont(\roman*)]
\crefalias{enumi}{proposition}
    \item \label{prop:scaled_pgd_convergence_i} $F(x_k)$ converges.
    \item \label{prop:scaled_pgd_convergence_ii} $\sum_{k=1}^{+\infty}\|x_k-x_{k+1}\|^2<+\infty$.
    \item \label{prop:scaled_pgd_convergence_iii} If $\xi(x,t) = F(x)+\tfrac{1}{2}t^2$ is KL, then $\sum_{k=1}^{+\infty}\|x_k-x_{k+1}\|<+\infty$, and $\{x_k\}_{k\geq 1}$ converges to a critical point $x^*$ of $F$.
\end{enumerate}
\end{proposition}\begin{proof}
Define $u_k \coloneqq x_k - \eta \nabla \fid (x_k)$ and $e_k \coloneqq \nabla G_k(w_{k+1})$. Since $x_{k+1} = \lpn(w_{k+1}) \in \prox_{\reg}(w_{k+1})$, we have $\nabla \reg(x_{k+1}) = w_{k+1}-x_{k+1}$. Therefore,
\begin{equation}
e_k = \gamma w_{k+1} + (1-\gamma)x_{k+1}-u_k = \gamma \nabla \reg(x_{k+1}) + x_{k+1}-u_k.
\end{equation}
Equivalently, $\gamma \nabla \reg(x_{k+1}) + x_{k+1}-(u_k + e_k)=0$, which is precisely the first-order optimality condition for $x_{k+1} \in \prox_{\gamma \reg}(u_k+e_k)$. Using $h = \eta$, $f = \fid$, $g = \tfrac{\gamma}{\eta}\reg$ for~\cite[Lemma 8]{sun2017convergence},

\begin{equation}\label{eq:inexact_F}
F (x_k)-F(x_{k+1}) \geq \frac{1-\eta L_{\mathbf{fid}}}{4}\|x_{k+1}-x_k\|^2 - \frac{1}{1-\eta L_{\mathbf{fid}}}\|e_k\|^2.
\end{equation}
Since $\|e_k\|\leq C/(k+1)^p$ with $p>3/2$, we have $\sum_{k=0}^\infty\|e_k\|^2<\infty$. Summing~\cref{eq:inexact_F} and using that $F$ is bounded below therefore yields \ref{prop:scaled_pgd_convergence_ii}. Define the tail error $r_k \coloneqq \frac{1}{1-\eta L_{\mathbf{fid}}}\sum_{l=k}^\infty\|e_l\|^2$ and the augmented sequence $\widetilde{F}_k \coloneqq F(x_k)+r_k$. Then~\cref{eq:inexact_F} becomes
\begin{equation}
\widetilde{F}_{k+1} \leq \widetilde{F}_k-\frac{1-\eta L_{\mathbf{fid}}}{4}\|x_{k+1}-x_k\|^2.
\end{equation} Since $F$ is bounded below, $\widetilde{F}_k$ is bounded below, so $\widetilde{F}_k\to\widetilde{F}^*$ for some $\widetilde{F}^*\in\R$. Since $\sum_{k=0}^\infty\|e_k\|^2<\infty$, we have $r_k\to0$. Therefore, $F(x_k)=\widetilde{F}_k-r_k\to\Tilde{F}^*$, which proves \ref{prop:scaled_pgd_convergence_i}.

We next note that~\cite[Proposition 1]{sun2017convergence} is stated under the assumption that $F$ is semi-algebraic. However, inspection of their proof shows that semi-algebraicity is only used to establish the KL property of the function $\xi(x,t)=F(x)+\tfrac12 t^2$. Thus, the same proof applies whenever $\xi$ is KL. Since $\|e_k\|\leq C/(k+1)^{p} = \mathcal{O}(1/k^{p})$ for $p>1$, $\sum_{k=0}^\infty\|x_{k+1}-x_k\|<\infty$, and $\{x_k\}_{k\geq 1}$ converges to a critical point $x^*$ of $F$. \ref{prop:scaled_pgd_convergence_iii} follows. 
\end{proof}

\section{Experiments}\label{section:experiments}
In this section, we demonstrate that proximal gradient descent using Learned Proximal Networks is competitive with modern approaches. The code to reproduce our experiments is available on GitHub.\footnote{\url{https://github.com/sparsity-group/pnplpn}} All training and testing were performed on a pair of NVIDIA RTX PRO 6000 Blackwell Max-Q GPUs using Ubuntu 24.04. 

In each experiment, the LPN uses a Huberized ReLU activation, given by
\begin{equation}\label{eq:huberized_relu}
H_\delta (t) \coloneqq
\begin{cases}
0, & t \leq 0,\\
\dfrac{t^2}{2\delta}, & 0<t<\delta,\\
t-\dfrac{\delta}{2}, & t\geq\delta.\\
\end{cases}
\end{equation}

\begin{figure}[t]
\centering
\begin{tikzpicture}
\begin{groupplot}[
    group style={
        group size=2 by 1,
        horizontal sep=1.2cm,
    },
    width=2.55in,
    height=1.95in,
    axis lines=middle,
    xlabel={$x$},
    samples=200,
    domain=-2:2,
]

\nextgroupplot[
    title={Huberized ReLU},
    ymin=-0.15,
    ymax=2.1,
    xmin=-2,
    xmax=2,
    ylabel={},
    legend style={
        draw=none,
        font=\small,
        at={(0.5,-0.075)},
        anchor=north,
        legend columns=-1,
        /tikz/every even column/.append style={column sep=0.35cm},
    },
]

\addplot[
    thick,
    dashed,
]
{max(0,x)};
\addlegendentry{ReLU}

\addplot[
    thick,
]
{x < 0 ? 0 :
 (x <= 0.5 ? x^2/(2*0.5) :
  x - 0.5/2)};
\addlegendentry{Huberized ReLU}

\nextgroupplot[
    title={Huber loss},
    ymin=-0.15,
    ymax=2.1,
    xmin=-2,
    xmax=2,
    ylabel={},
    legend style={
        draw=none,
        font=\small,
        at={(0.5,-0.075)},
        anchor=north,
        legend columns=-1,
        /tikz/every even column/.append style={column sep=0.35cm},
    },
]

\addplot[
    thick,
    dashed,
]
{abs(x)};
\addlegendentry{$|x|$}

\addplot[
    thick,
    dotted,
]
{0.5*x^2};
\addlegendentry{$\frac12 x^2$}

\addplot[
    thick,
]
{abs(x) <= 1 ? 0.5*x^2 : abs(x)-0.5};
\addlegendentry{Huber}

\end{groupplot}
\end{tikzpicture}

\caption{ The Huberized ReLU compared to Huber loss. Both replace the nonsmooth kink at the origin with a quadratic transition band.}
\label{fig:huberized_relu}
\end{figure}
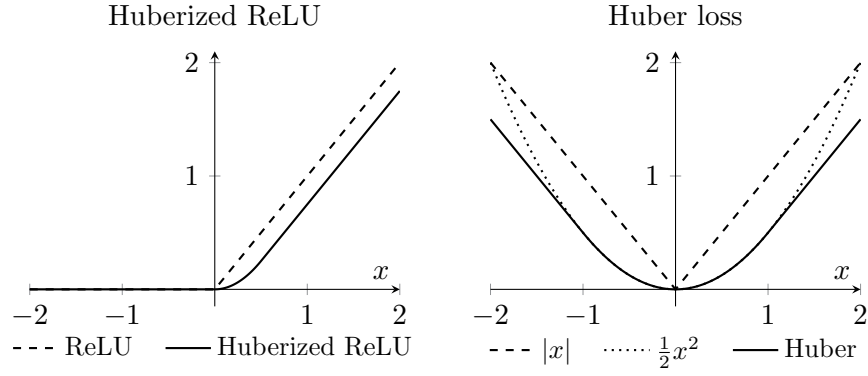

All LPNs follow the architecture introduced in~\cref{def:LPN}, with exact architectures summarized in~\cref{tab:lpn_architectures}. Training details are summarized in~\cref{tab:lpn_training}, including links to the relevant datasets.

\begin{table}[htb!]
\centering
\small
\caption{LPN architectural details for the inverse problems experiments.}
\label{tab:lpn_architectures}
\setlength{\tabcolsep}{4pt}
\renewcommand{\arraystretch}{1.1}
\begin{tabularx}{\textwidth}{
    >{\raggedright\arraybackslash}p{0.17\textwidth}
    >{\raggedright\arraybackslash}X
    >{\raggedright\arraybackslash}X}
\toprule
& \textbf{MRI ($320\times320$)}
& \textbf{Computed Tomography ($512\times512$)} \\
\midrule
\textbf{Feature dimensions}
& $H\times160^2$, $H\times160^2$,
  $H\times80^2$, $H\times80^2$,
  $H\times40^2$, $H\times20^2$,
  $H\times20^2$, $H\times1^2$
& $H\times256^2$, $H\times256^2$,
  $H\times128^2$, $H\times128^2$,
  $H\times64^2$, $H\times32^2$,
  $H\times32^2$, $H\times1^2$ \\
\textbf{Skip connections}
& $5$ input-to-feature convolutions;
  $3\times3$, stride $1$, after bilinear resizing
& $5$ input-to-feature convolutions;
  kernels $3,4,3,4,4$;
  strides $1,2,1,2,2$ \\
\midrule
\textbf{Architecture}
& \multicolumn{2}{>{\raggedright\arraybackslash}p{0.75\textwidth}}{
Hidden dimension $H=256$, $L=7$ layers,
with convolutional kernels $4,3,4,3,4,4,3$ and
strides $2,1,2,1,2,2,1$, followed by global average pooling
and a fully connected layer.
} \\
\textbf{Parameters}
& \multicolumn{2}{>{\raggedright\arraybackslash}p{0.75\textwidth}}{
Huberized-ReLU activation with $\delta=0.1$ and $\alpha=0.8$,
step size $\eta=0.99$.
} \\
\bottomrule
\end{tabularx}
\end{table}
\begin{table}[htb!]
\centering
\small
\caption{Training procedures for the inverse problems experiments.
All images are normalized to $[0,1]$.}
\label{tab:lpn_training}
\setlength{\tabcolsep}{4pt}
\renewcommand{\arraystretch}{1.08}
\begin{tabularx}{\textwidth}{
    >{\raggedright\arraybackslash}p{0.17\textwidth}
    >{\raggedright\arraybackslash}X
    >{\raggedright\arraybackslash}X}
\toprule
& \textbf{MRI}
& \textbf{Computed Tomography} \\
\midrule
\textbf{Dataset}
& fastMRI Multicoil Knee
& Mayo Clinic CT \\
\textbf{Input size}
& $1\times320\times320$
& $1\times512\times512$ \\
\midrule
\textbf{Optimization}
& \multicolumn{2}{>{\raggedright\arraybackslash}p{0.75\textwidth}}{
Adam with cosine annealing from an initial learning rate of $10^{-4}$
to a final learning rate of $10^{-6}$. $40{,}000$ training steps.
Effective batch size of $64$.
} \\
\textbf{Validation}
& \multicolumn{2}{>{\raggedright\arraybackslash}p{0.75\textwidth}}{
$16$ validation images randomly selected before training, evaluated every
$2{,}500$ training steps.
} \\
\bottomrule
\end{tabularx}
\end{table}

We use the clean reference images provided by each dataset, with pixel values normalized to $[0,1]$. The LPNs were trained on pairs ${(\bar x_i,z_i)}$, where $\bar x_i$ is a clean image and $z_i=\bar x_i+\epsilon_i$, with $\epsilon_i$ independent Gaussian noise of variance $0.1$. Each network was trained with a batch size of 64 under mean-squared error loss
\begin{equation}
\mathcal{L}(\theta)
\coloneqq
\frac{1}{64}\sum_{i=1}^{64}
\|f_{\theta,\alpha}(z_i)-\bar x_i\|^2,
\end{equation}
 The strong convexity modulus $\alpha$ was fixed at $0.8$. A selected subset of 16 images from each dataset was reserved for testing and not used during training. For testing, each clean test image $\bar x_i$ is used to generate a measurement $y_i = A \bar x_i + \epsilon_i,$
where $A$ is the forward operator associated with the corresponding inverse problem and $\epsilon_i$ denotes additive measurement noise. In all cases, reconstruction was performed using the least-squares data-fidelity term $
\fid_i(x) \coloneqq \tfrac{1}{2} \|Ax-y_i\|^2.$ In each inverse problem, we evaluate the three LPN-based reconstruction schemes discussed in~\cref{section:convergence_analysis}. First, we apply standard LPN-PGD (\cref{alg:lpn_pgd_vanilla}), using the trained LPN directly. Secondly, we apply averaged LPN-PGD, where the LPN $\lpn$ is replaced with an averaged LPN $\lpn^{(\lambda)} \coloneqq  (1-\lambda)I +\lambda \lpn$, for some $\gamma$ between $0$ and $1$ selected based on empirical reconstruction performance. Finally, we evaluate inexact scaled LPN-PGD (\cref{alg:lpn_pgd_scaled}), which directly scales the implicit regularizer by some $\gamma$ between $0$ and $1$.

\subsection{Magnetic Resonance Imaging (MRI)}
We first consider accelerated multicoil MRI using the fastMRI multicoil knee dataset. The experiments are performed only on the clean images, but the $k$-space data is used for the purpose of estimating the $C=15$ coil sensitivities. For each test image $\bar x_i$, and each coil $c$, we estimate the sensitivity map $S_c$, indexed in two dimensions by $k$ and $j$ with the following procedure. We load the corresponding fully sampled multicoil $k$-space $K_{i,c}$ and retain only the 30 central phase-encoding lines of $k$-space, zeroing the remaining lines, to obtain $K_{i,c}^{\mathrm{cal}}$, the calibration $k$-space. Next, we apply the inverse Fourier transform to obtain a low-resolution image for each coil: $\tilde{x}_{i,c} = \mathcal{F}^{-1}(K_{i,c}^{\mathrm{cal}})$. Each $\tilde{x}_{i,c}$ is approximately a low-resolution $\bar{x}_i$ multiplied by coil $c$'s sensitivity: at a given coordinate $(k,j)$, $\tilde{x}_{i,c}[k,j] \approx (S_c\bar{x}_i^{\mathrm{low}})[k,j].$ We use $\hat{S}_c[k,j] = \tilde{x}_{i,c}[k,j]/R[k,j]$ as the corresponding coil sensitivity, where $R[k,j] = \sqrt{\sum_{c=1}^{C} |\tilde{x}_{i,c}[k,j]|^2}$ is the RSS magnitude across coils of the low-resolution approximation at that pixel. 

The resulting forward operator is $A(x) \coloneqq \begin{bmatrix} M\mathcal{F}(\hat{S}_1 x) &
\dots & M\mathcal{F}(\hat{S}_C x) \end{bmatrix}^T,$
where $\mathcal{F}$ is a two-dimensional Fourier transform, $\hat{S}_c$ is the estimated sensitivity map for coil $c$, and $M$ is a Cartesian sampling mask. The sampling mask is parametrized by an acceleration factor $R$, which determines the fraction of phase-encoding lines retained, and a center fraction $f_{\mathrm{center}}$, which specifies the fraction of central phase-encoding lines that are always sampled. For example, with a mask where $R=4$ and $f_{\mathrm{center}} = 0.1$, $25\%$ of the phase-encoding lines are retained in total, with the central $10\%$ included. The remaining frequencies sampled randomly at a rate of about $\frac{15\%}{90\%} = 16.7\%$.  

The adjoint is explicit. For multicoil data $z=\begin{bmatrix}
    z_1 & \dots & z_C
\end{bmatrix}$, $A^*z = \sum_{c=1}^{C}\widebar{\hat{S}_c}\mathcal{F}^{-1}(Mz_c)$, where
$A^*y$ is called the \emph{zero-filled} reconstruction because the unmeasured Fourier coefficients are implicitly set to zero before applying the inverse Fourier transform. The forward model has operator norm at most $1$ since it uses RSS-normalized coil maps and binary sampling masks. We use a step size of $\eta=0.99$ for all MRI experiments. \cref{fig:mri_reconstruction_vanilla2,fig:mri_reconstruction_vanilla4} show reconstructions obtained using standard LPN-PGD (\cref{alg:lpn_pgd_vanilla}) at acceleration factors $R=2$ and $R=4$, respectively. \cref{fig:mri_reconstruction_3averaged2,fig:mri_reconstruction_3averaged4} show the corresponding reconstructions using averaged LPN-PGD with averaging parameter $\gamma=0.3$, while \cref{fig:mri_reconstruction_3gamma2,fig:mri_reconstruction_3gamma4} show reconstructions obtained using inexact scaled LPN-PGD (\cref{alg:lpn_pgd_scaled}) with regularizer-scaling parameter $\gamma=0.3$. Each MRI reconstruction is displayed after 100 iterations. 
\FloatBarrier
\begin{figure}[p!]
    \centering
     \includegraphics[width=.85\linewidth]{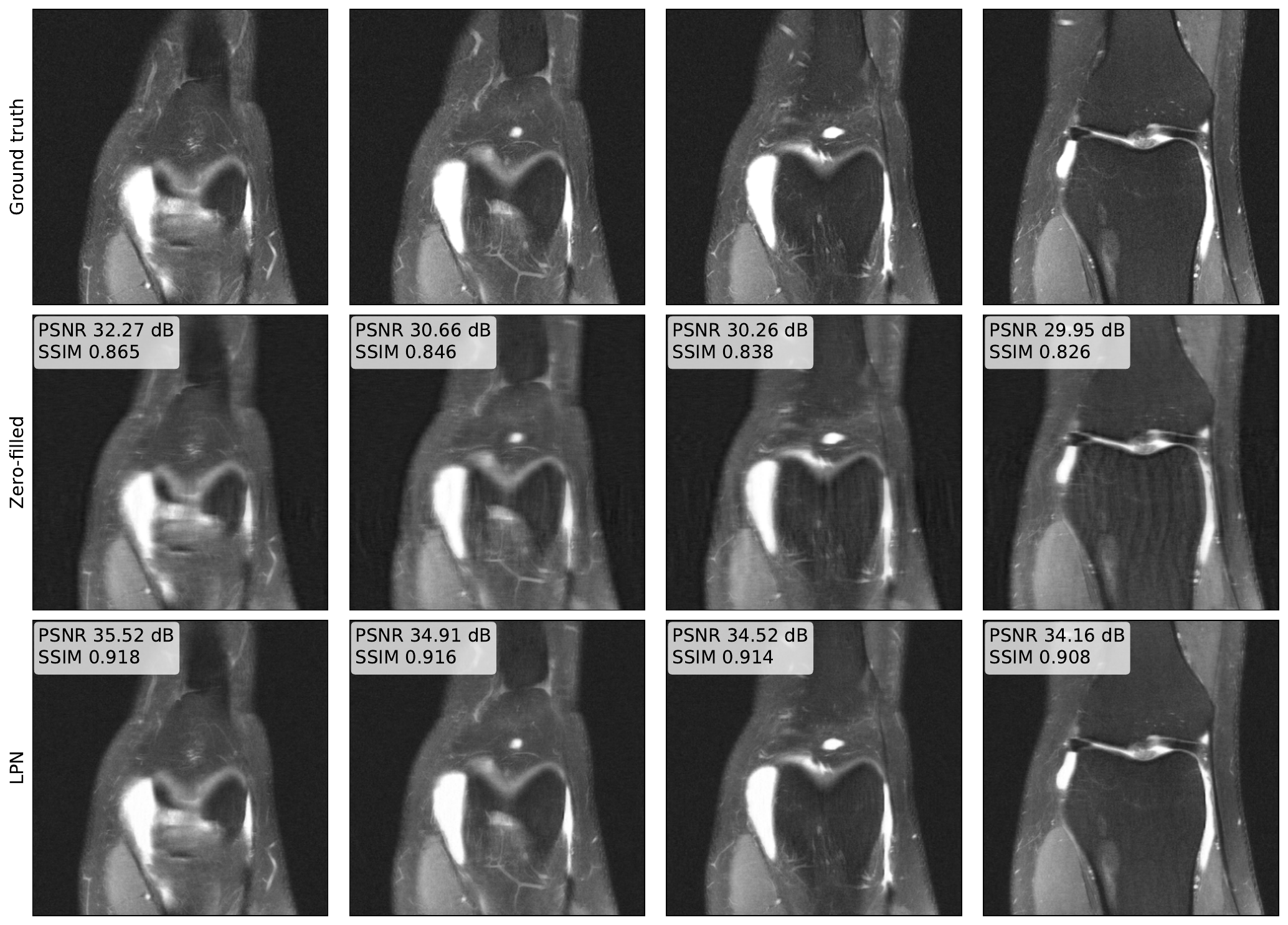}
    \caption{Standard LPN-PGD reconstruction after $100$ iterations with $R=2$ and $f_{\mathrm{center}}=0.1$.}
    \label{fig:mri_reconstruction_vanilla2}
    \centering
    \includegraphics[width=.85\linewidth]{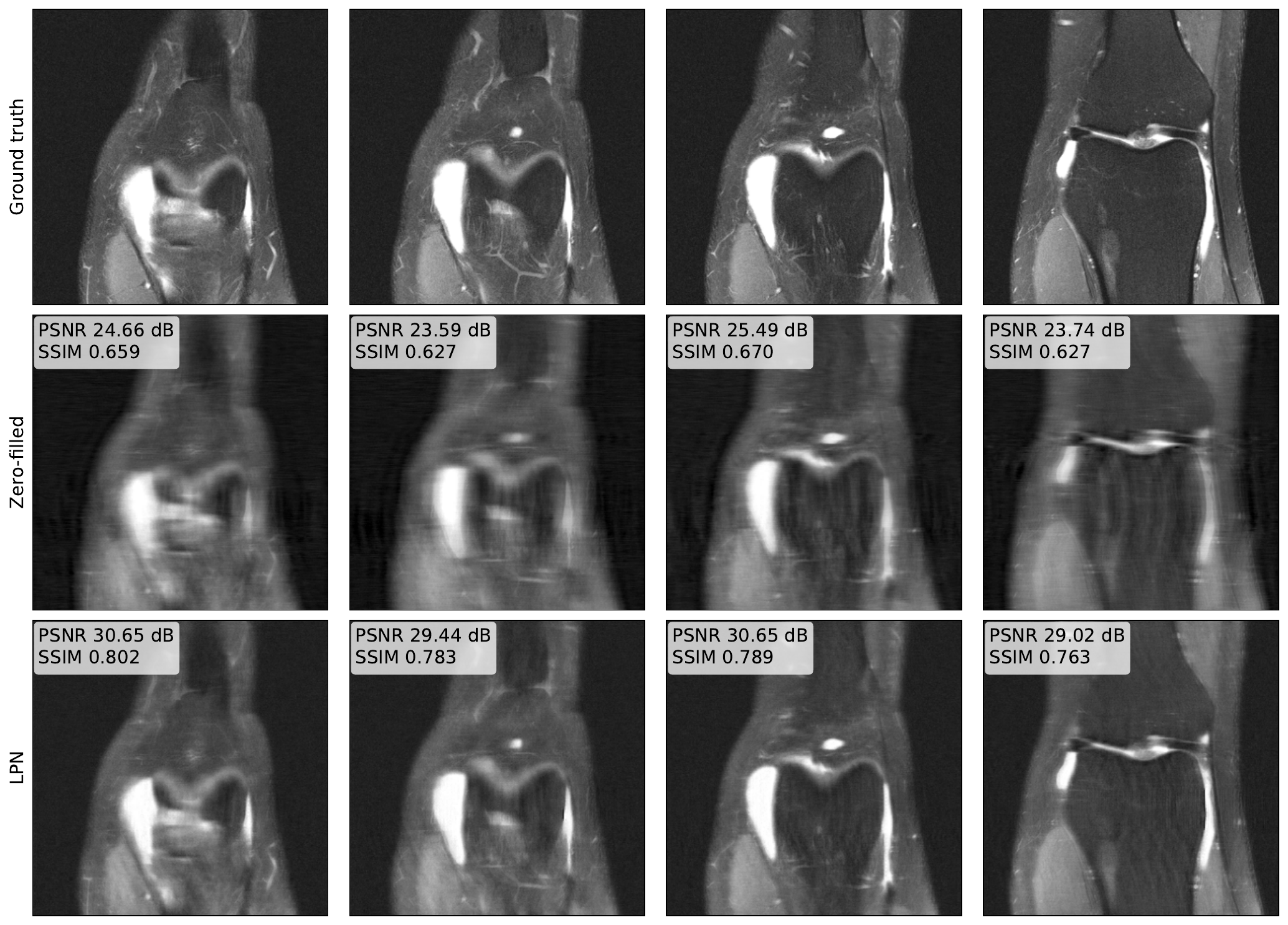}
    \caption{Standard LPN-PGD reconstruction after $100$ iterations with $R=4$ and $f_{\mathrm{center}}=0.05$.}
    \label{fig:mri_reconstruction_vanilla4}
\end{figure}
\begin{figure}[p!]
    \centering
    \includegraphics[width=.85\linewidth]{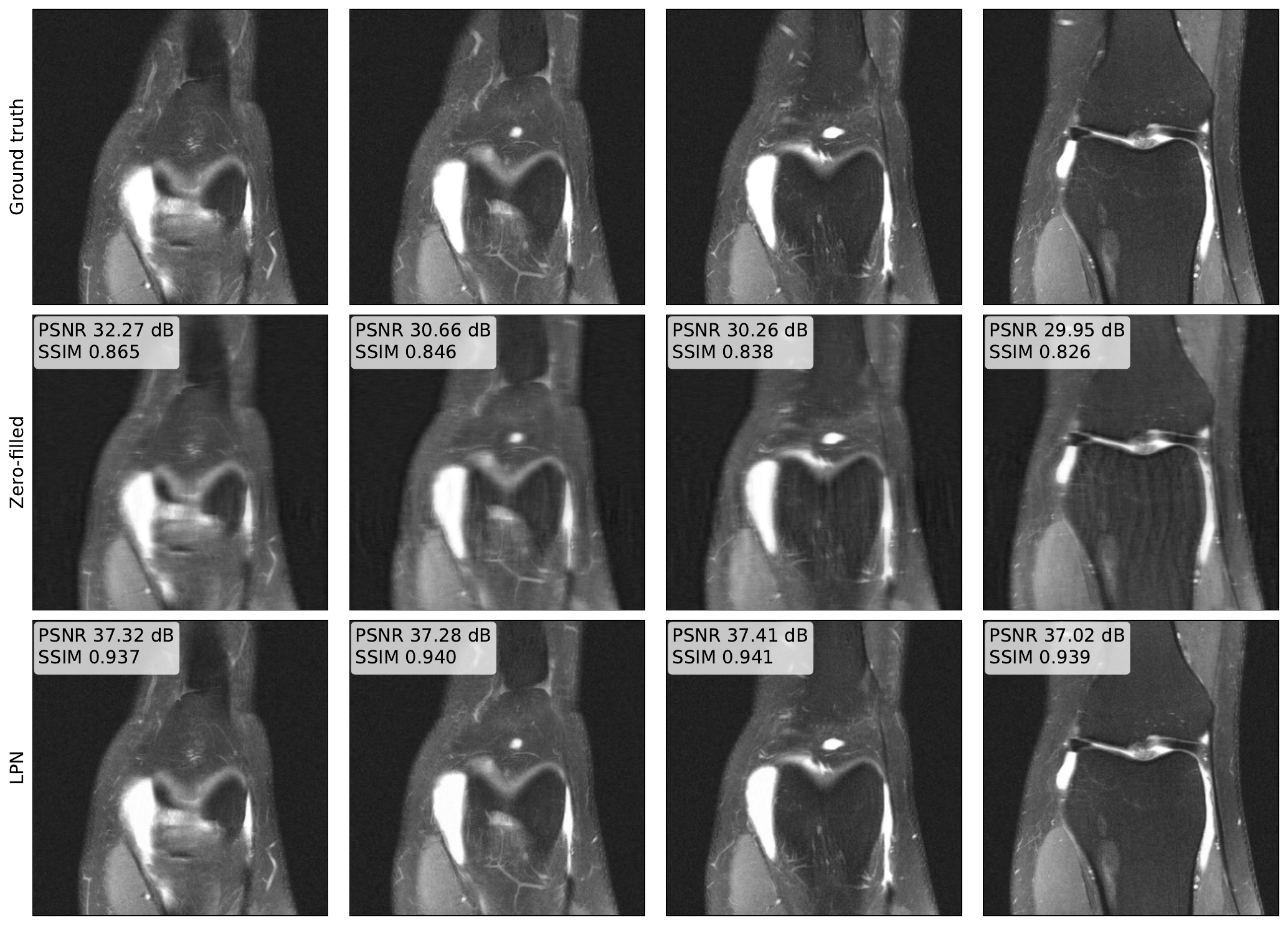}
    \caption{Averaged LPN-PGD reconstruction after $100$ iterations with $\gamma=0.3$, $R=2$, $f_{\mathrm{center}}=0.1$.}
    \label{fig:mri_reconstruction_3averaged2}
    \centering
    \includegraphics[width=.85\linewidth]{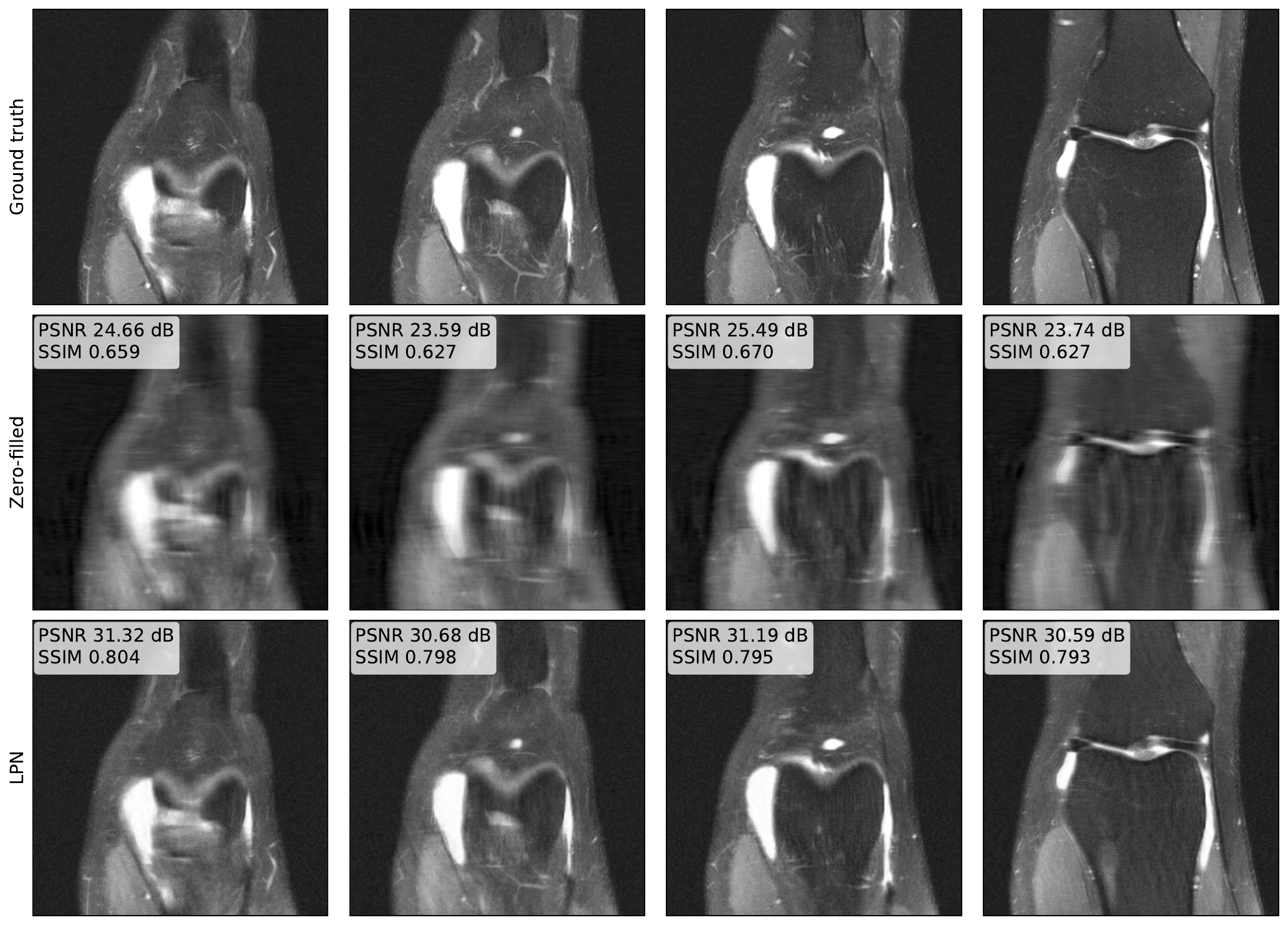}
    \caption{Averaged LPN-PGD reconstruction after $100$ iterations with $\gamma=0.3$, $R=4$, $f_{\mathrm{center}}=0.05$.}
    \label{fig:mri_reconstruction_3averaged4}
\end{figure}
\begin{figure}[p!]
    \centering
    \includegraphics[width=.85\linewidth]{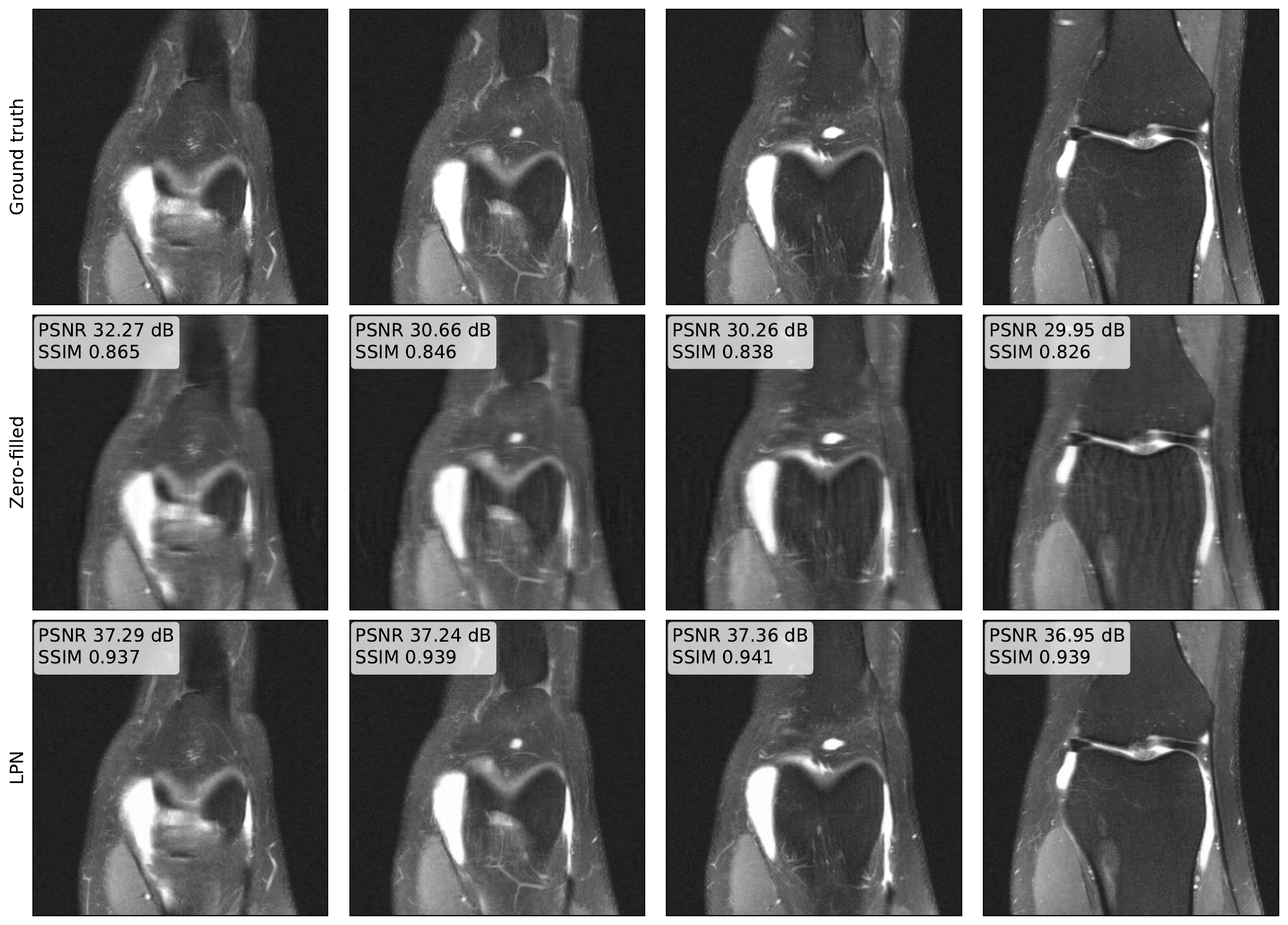}
    \caption{Scaled LPN-PGD reconstruction after $100$ iterations with $\gamma=0.3$, $R=2$, $f_{\mathrm{center}}=0.1$.}
    \label{fig:mri_reconstruction_3gamma2}
    \includegraphics[width=.85\linewidth]{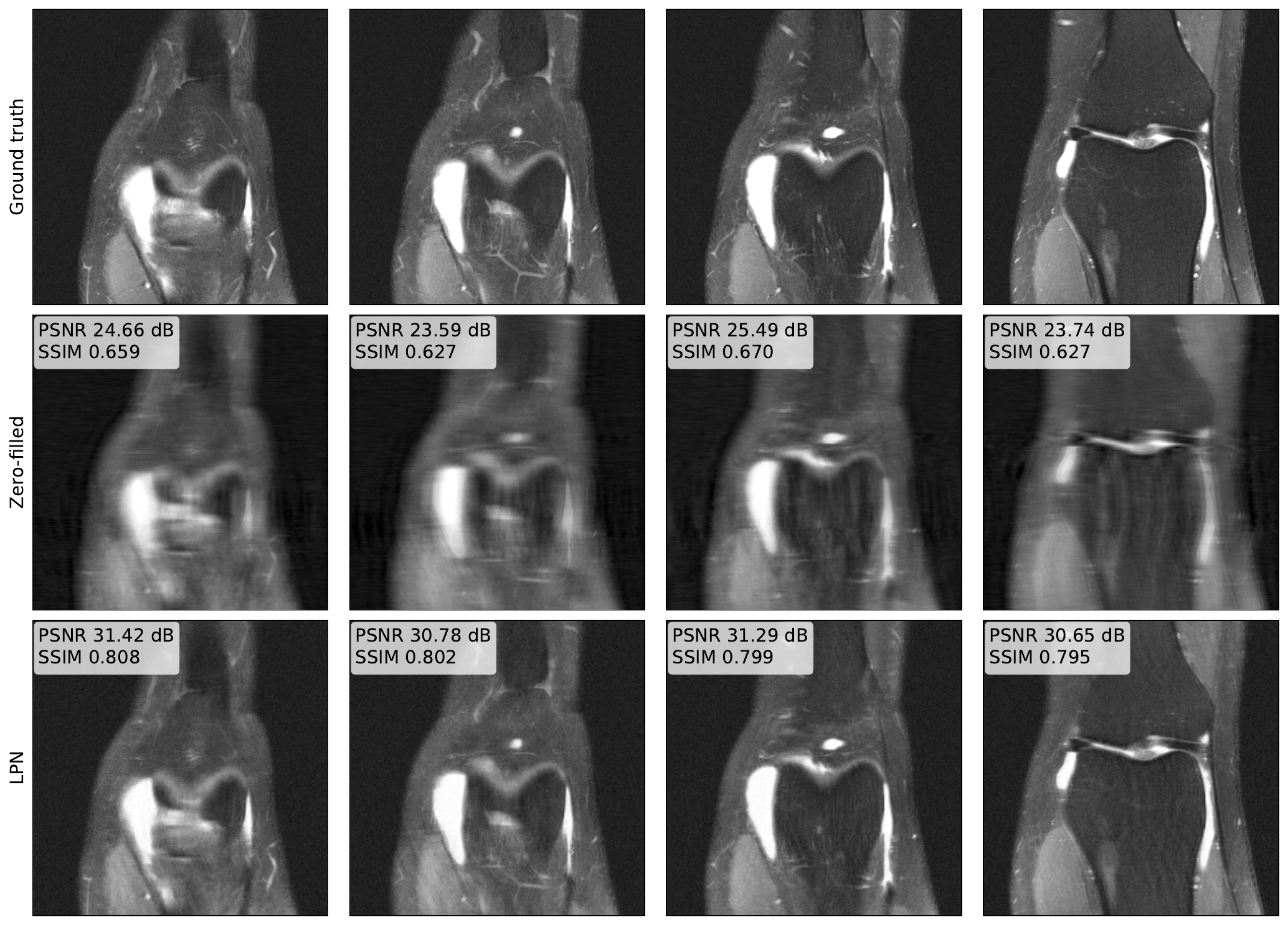}
    \caption{Scaled LPN-PGD reconstruction after $100$ iterations with $\gamma=0.3$, $R=4$, $f_{\mathrm{center}}=0.05$.}
    \label{fig:mri_reconstruction_3gamma4}
\end{figure}
\FloatBarrier

\subsection{Computed Tomography (CT)}
We next consider X-Ray computed tomography using the MayoCT dataset, where the forward operator $A$ is a parallel-beam Radon transform, implemented using the LEAP library~\cite{kim2023differentiable}. The measurement geometry is specified by $N_\theta$ projection angles and $N_d=400$ detector bins. We vary $N_\theta$ and the measurement-noise level to consider two regimes: low-dose CT, where the noise level is high (variance of $1$) but $N_\theta$ remains large ($180$ angles), and sparse-view CT, where the noise level is low (variance of $0.1$) but $N_\theta$ is substantially reduced ($60$ angles). In each case, we estimate the operator norm of $A$ using power iterations up to tolerance $10^{-6}$ for the purpose of normalizing the forward operator to operator norm of $1$. We display the original image, the filtered backprojection (FBP), and the reconstructed image. \cref{fig:ct_reconstruction_lowdose_vanilla,fig:ct_reconstruction_sparseview_vanilla} show the low-dose and sparse-view reconstructions obtained using standard LPN-PGD (\cref{alg:lpn_pgd_vanilla}), respectively. Figures~\cref{fig:ct_reconstruction_lowdose_averaged,fig:ct_reconstruction_sparseview_averaged} show the corresponding reconstructions using averaged LPN-PGD with $\gamma=0.15$. Finally, Figures~\cref{fig:ct_reconstruction_lowdose_gamma,fig:ct_reconstruction_sparseview_gamma} show the low-dose and sparse-view reconstructions obtained using inexact scaled LPN-PGD (\cref{alg:lpn_pgd_scaled}) with scaling parameter $\gamma=0.15$.

\begin{figure}[p!]
    \centering
    \includegraphics[width=.85\linewidth]{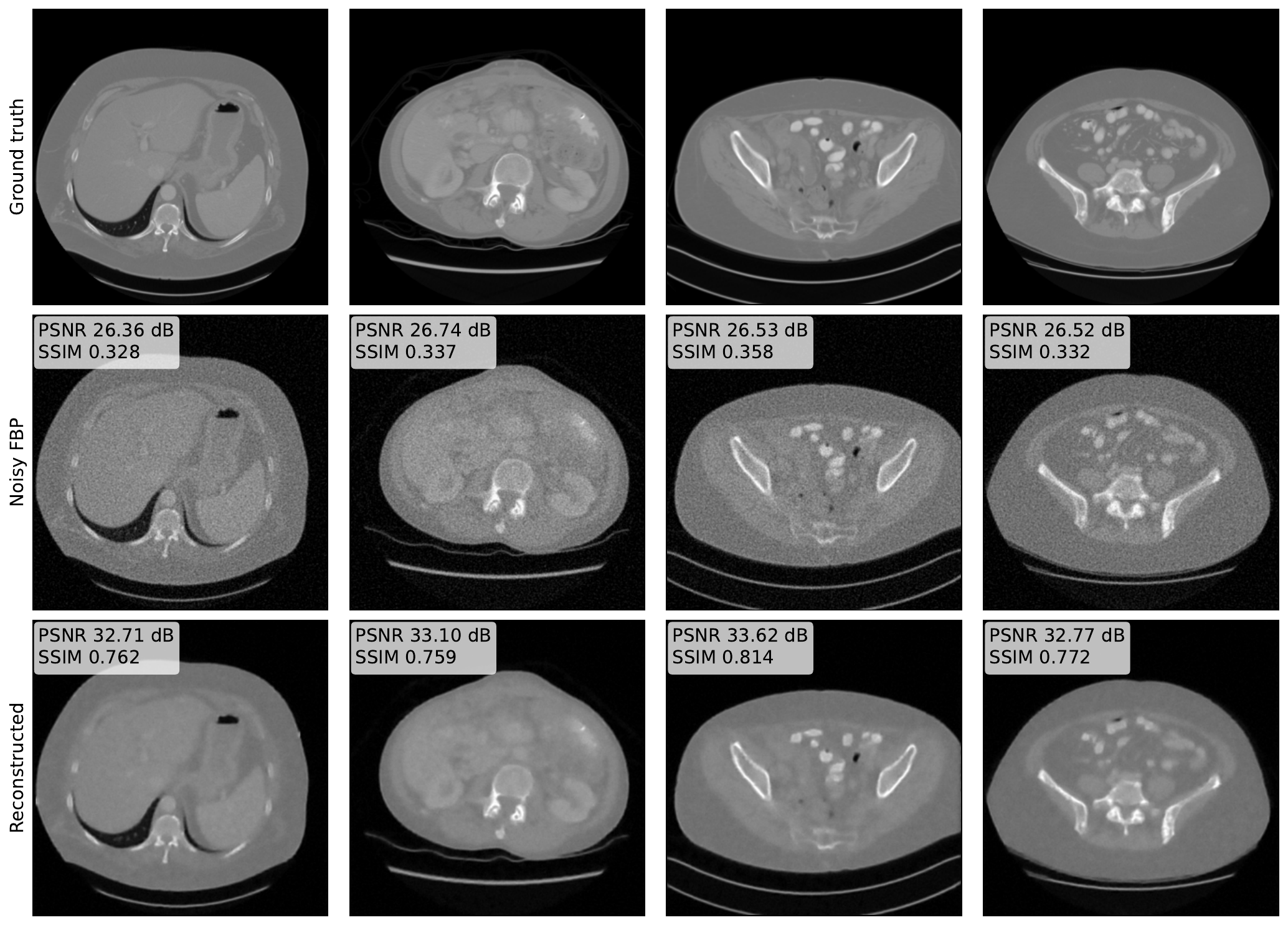}
    \caption{Low-dose CT using LPN-PGD for $10$ iterations with $N_\theta=180$.}
    \label{fig:ct_reconstruction_lowdose_vanilla}
    \centering
    \includegraphics[width=.85\linewidth]{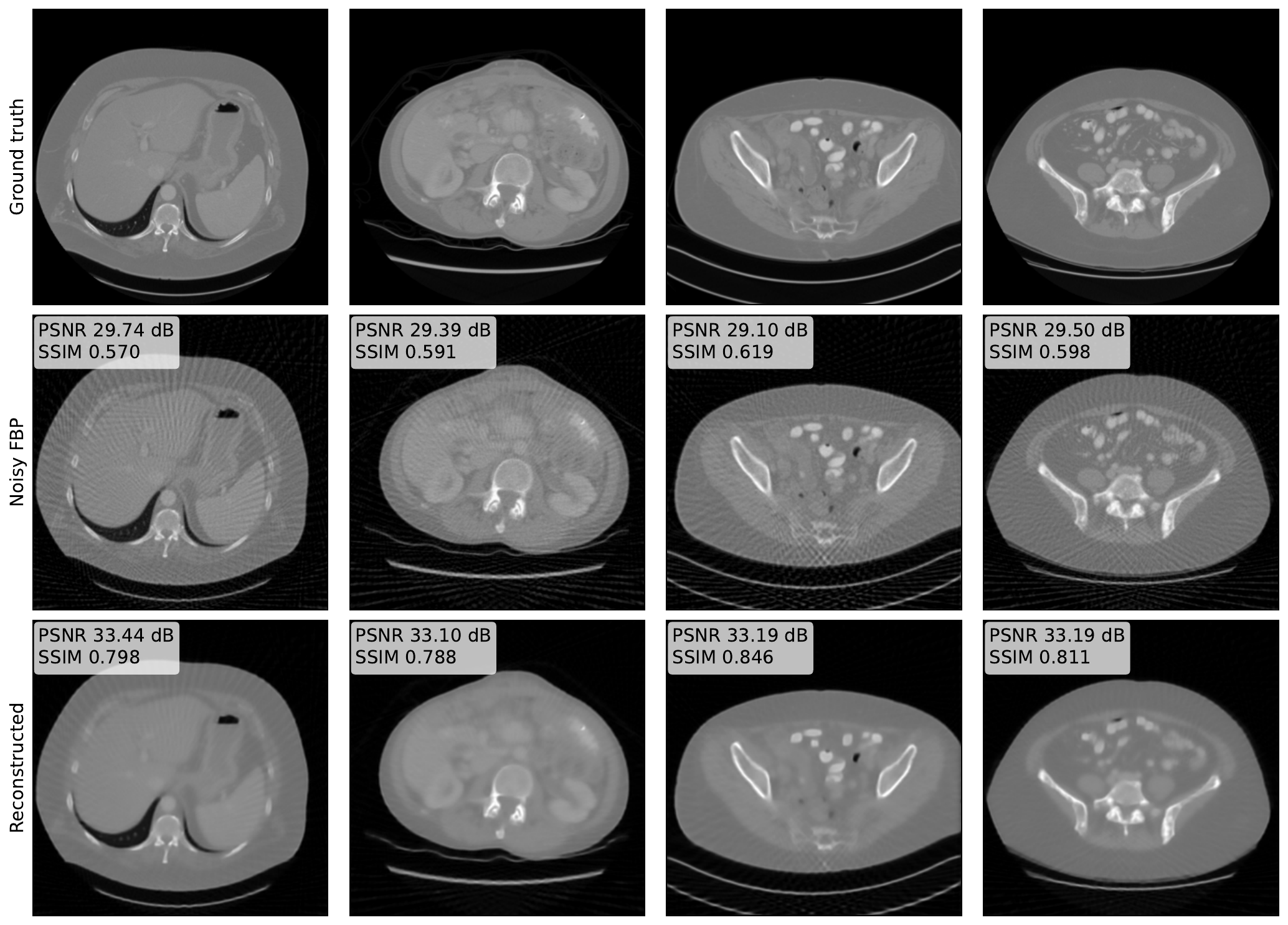}
    \caption{Sparse-view CT using LPN-PGD for $10$ iterations with $N_\theta=60$.}

    \label{fig:ct_reconstruction_sparseview_vanilla}
\end{figure}
\begin{figure}[p!]
    \centering
    \includegraphics[width=.85\linewidth]{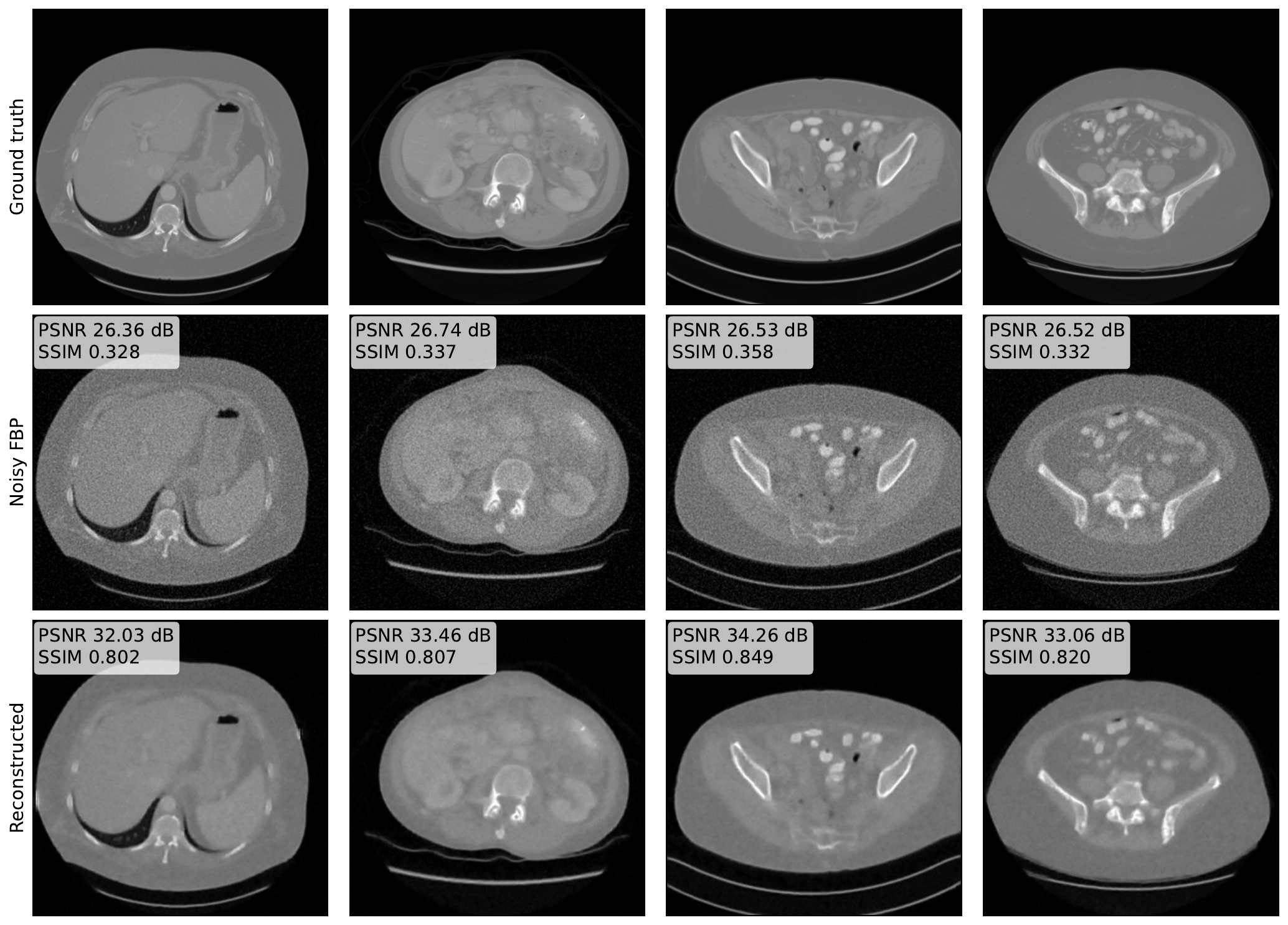}
    \caption{Low-dose CT using averaged LPN-PGD for $100$ iterations with $\gamma=0.15$ and $N_\theta=180$.}
    \label{fig:ct_reconstruction_lowdose_averaged}
    \centering
    \includegraphics[width=.85\linewidth]{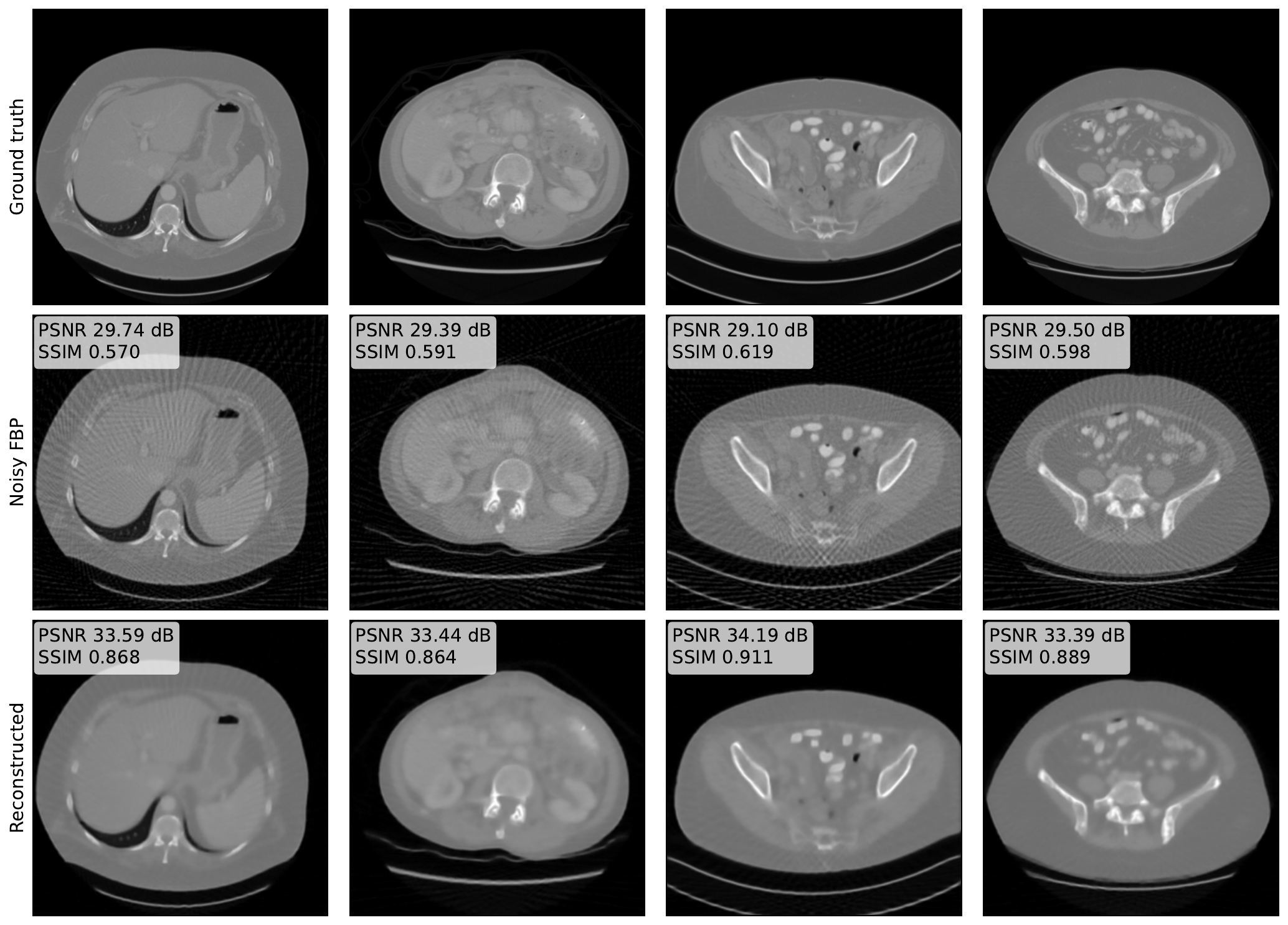}
    \caption{Sparse-view CT using averaged LPN-PGD for $100$ iterations with $\gamma=0.15$ and $N_\theta=60$.}
    \label{fig:ct_reconstruction_sparseview_averaged}
\end{figure}
\begin{figure}[p!]
    \centering
    \includegraphics[width=.85\linewidth]{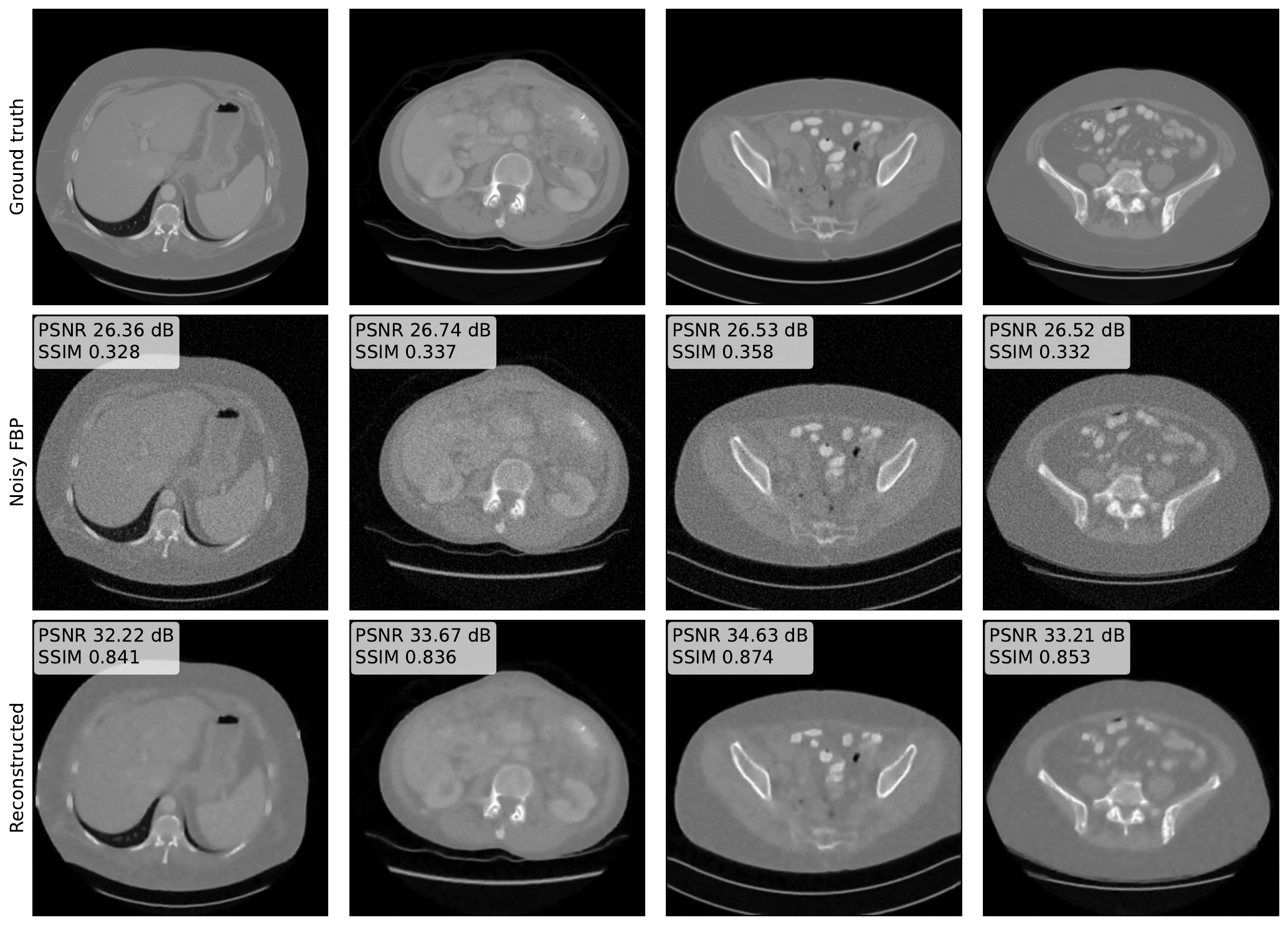}
    \caption{Low-dose CT using scaled LPN-PGD for $100$ iterations with $\gamma=0.15$ and $N_\theta=180$.}
    \label{fig:ct_reconstruction_lowdose_gamma}

    \includegraphics[width=.85\linewidth]{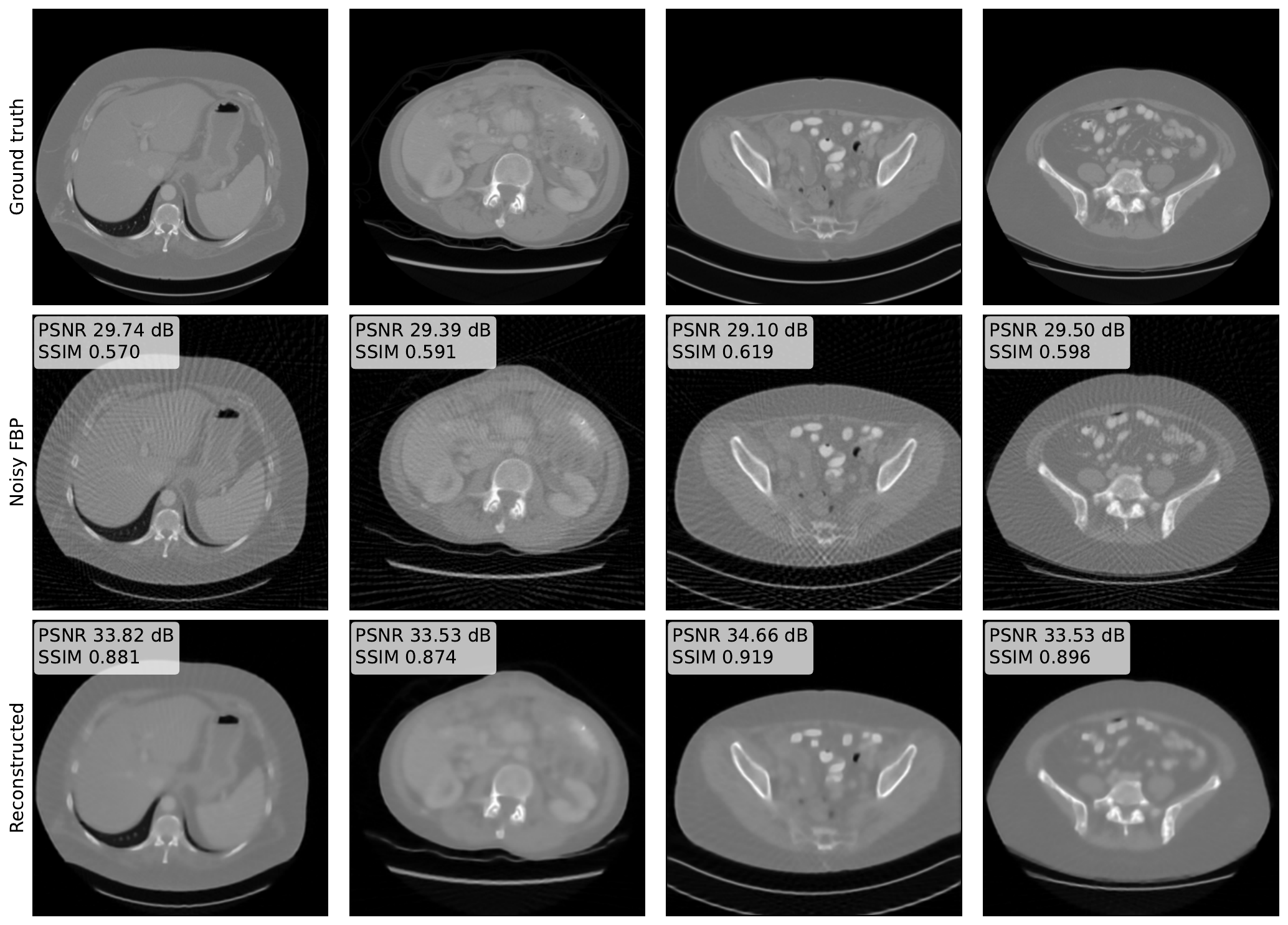}
    \caption{Sparse-view CT using scaled LPN-PGD for $100$ iterations with $\gamma=0.15$ and $N_\theta=60$.}
    \label{fig:ct_reconstruction_sparseview_gamma}
\end{figure}

\section{Conclusion and Future Work} \label{section:conclusion}
In this work, we revisited the Learned Proximal Network (LPN) framework, where a learned denoiser is guaranteed to be the proximal operator of a well-behaved regularizer by construction. We clarified the mathematical structure underlying LPNs, extended the framework to a broader class of activation functions, and established convergence guarantees for the resulting proximal algorithms. We also developed two mechanisms for weakening the strength of the learned prior. First, we showed that averaging an LPN with the identity preserves the LPN structure and corresponds to replacing the implicit regularizer by a scaled Moreau-smoothed version. Second, we developed a procedure for directly scaling the implicit regularizer by evaluating $\prox_{\gamma\reg}$ through a strongly convex auxiliary problem, and showed that this inner problem may be solved inexactly while retaining convergence guarantees. Finally, we demonstrated the framework on full-resolution accelerated MRI and CT reconstruction experiments, which achieved competitive reconstruction quality relative to state-of-the-art results. 

The experiments also exposed several questions not resolved by the present theory. In particular, the learned prior can remain excessively strong even after averaging or direct scaling. In the CT experiments, reconstruction quality eventually deteriorates after too many iterations, whereas we do not observe the same behavior in the MRI experiments. It is unclear whether this discrepancy is caused primarily by the learned prior, the denoising training procedure, the forward operator, or the interaction between these factors. Understanding this behavior is especially important because the convergence guarantees developed here ensure convergence to a critical point of the learned objective, but do not by themselves guarantee that this critical point corresponds to the best reconstruction.

Though the ICNN architecture provides a convenient way to guarantee convexity and obtain an exact proximal interpretation, it has been shown to be restrictive in its representation of convex functions~\cite{gagneux2025convexity}. This may limit the class of proximal operators that an LPN can learn. Investigating the expressivity of LPNs, as well as other possible parameterizations, is therefore an important direction for future work. More broadly, this work lays the groundwork for further study of LPNs as a promising class of learned reconstruction methods. A necessary first step is to establish that LPNs fit naturally within the classical framework of proximal methods and satisfy the regularity and convergence properties required for well-behaved optimization. With this foundation in place, questions of expressivity, training, and reconstruction performance can be studied more systematically.

\putbib[ref]
\end{bibunit}
\clearpage

\begin{center}
\color{header1}
{\bfseries\MakeUppercase{Supplementary Materials}:}
{\HLtitle
Plug-and-Play Methods Provably Converge Even with Improperly Trained Denoisers:
Convergence by Architectural Design
\par}

\vspace{0.075in}
\color{gray}\rule{\textwidth}{4pt}
\end{center}

\vspace{0.11in}

\setcounter{section}{0}
\setcounter{subsection}{0}
\setcounter{equation}{0}
\setcounter{figure}{0}
\setcounter{table}{0}
\setcounter{algorithm}{0}
\setcounter{theorem}{0}
\setcounter{footnote}{0}
\makeatletter
\def\siamprelabel{SM}
\renewcommand{\@biblabel}[1]{[SM#1]}

\renewcommand{\thesection}{SM\arabic{section}}
\renewcommand{\thesubsection}{\thesection.\arabic{subsection}}

\renewcommand{\theHsection}{SM.\arabic{section}}
\renewcommand{\theHsubsection}{SM.\arabic{section}.\arabic{subsection}}
\renewcommand{\theHequation}{SM.\arabic{equation}}
\renewcommand{\theHfigure}{SM.\arabic{figure}}
\renewcommand{\theHtable}{SM.\arabic{table}}
\renewcommand{\theHalgorithm}{SM.\arabic{algorithm}}
\renewcommand{\theHtheorem}{SM.\arabic{theorem}}

\begin{bibunit}[siamplain]

\section{The KL Property of LPN Objectives} 
\label{section:KL_objective}
The convergence results presented in~\cref{section:convergence_analysis} are greatly enhanced when the objective is KL (see~\cref{def:KL}). In this section, we prove~\cref{prop:lpn_objective_KL}. Informally, this result states that when the activation of the LPN and data fidelity term in question are constructed from common elementary and analytic functions using finitely many standard operations, the resulting objective is KL. We defer the proof to the end of this section, and first develop the machinery needed to establish the result. In particular, we introduce the notion of o-minimal structures and definable functions~\cite{tarski1998decision, seidenberg1954new, gabrielov1968projections, van1984remarks, pillay1986definable, knight1986definable, denef1988p, van1994elementary, wilkie1996model, van1996geometric}. 
\begin{definition}[o-minimal structure]
Let $\mathcal{M} \coloneqq \bigcup_{n\in\mathbb{N}} \mathcal{M}_n,$ where each $\mathcal{M}_n$ is a family of subsets of $\mathbb{R}^n$.
We say that $\mathcal{M}$ is an \emph{o-minimal structure} on
$(\mathbb{R},+,\cdot)$ if the following conditions hold:
\begin{enumerate}[
    label=\normalfont(\roman*),
]
    \item Each $\mathcal{M}_n$ is closed under finite set-theoretic
    operations.

    \item If $A\in\mathcal{M}_n$ and $B\in\mathcal{M}_m$, then
    $A\times B\in\mathcal{M}_{n+m}$.

    \item If $A\in\mathcal{M}_{n+m}$ and
    $\pi:\mathbb{R}^{n+m}\to\mathbb{R}^n$ is the projection onto the
    first $n$ coordinates, then $\pi(A)\in\mathcal{M}_n$.

    \item If $f,g_1,\ldots,g_k\in\mathbb{Q}[X_1,\ldots,X_n]$, then
    \[
    \left\{
    x\in\mathbb{R}^n:
    f(x)=0,\;
    g_1(x)>0,\ldots,g_k(x)>0
    \right\}
    \in\mathcal{M}_n.
    \]
    \item The elements of $\mathcal{M}_1$ are precisely the finite unions
    of points and open intervals.
\end{enumerate}
We say a function $f:\R^n \to \R$ is $\mathcal M$-definable if its graph, which we denote
$\graph f$, is in $\mathcal M$.
\end{definition}

See~\cite{kurdyka1998gradients} or \cite{ji2020directional} for basic properties and examples of o-minimal structures, as well as definable functions. In particular, sums, products, inverses, maxima, minima, and compositions of $\mathcal M$-definable functions are again $\mathcal M$-definable. We next show that $C^1$ functions definable in an o-minimal structure are KL. 

\begin{lemma}[Definable $C^1$ functions are KL]\label{lem:definable_KL}
Let $\mathcal M$ be an o-minimal structure, and let
$f:\R^n\to\R$ be continuously differentiable and $\mathcal{M}$-definable. Then $f$ is a KL function.
\end{lemma}

\begin{proof}
Fix $x^*\in\R^n$ and choose $r>0$. Let $U \coloneqq B(x^*,r)$ and define
\begin{equation}
    V
    \coloneqq
    U\cap\left\{x\in\R^n:f(x)>f(x^*)\right\}.
\end{equation}
Since $f$ is continuous and $\mathcal M$-definable, $V$ is an open, bounded, $\mathcal{M}$-definable subset of $\R^n$.

If $V$ is empty, then $f$ trivially has the KL property at $x^*$. Therefore, suppose that $V$ is not empty, and define $h:V\to\R$ and $h(x)\coloneqq f(x)-f(x^*)$.
The function $h$ is $C^1$ and definable, and positive on $V$. 

By Kurdyka's Inequality (\cite[Theorem 1]{kurdyka1998gradients}), there exist constants $c>0$ and $\rho>0$, and a strictly increasing positive $\mathcal M$-definable $C^1$ function $\Psi:[0,\infty)\to\R$ such that
\begin{equation}
    \left\|\nabla(\Psi\circ h)(x)\right\|
    \geq c
\end{equation}
for every $x\in V$ where $0<h(x)<\rho$. Therefore, by the chain rule,
\begin{equation}\label{eq:kurdyka_intermediate}
    \Psi'\bigl(f(x)-f(x^*)\bigr)\|\nabla f(x)\|
    \geq c
\end{equation}
whenever $x\in U$ and $f(x^*)<f(x)<f(x^*)+\rho.$
Because $\Psi$ is $\mathcal M$-definable, its derivative $\Psi'$ is $\mathcal M$-definable by~\cite[Lemma~3]{kurdyka1998gradients}. Applying the o-minimal
monotonicity theorem~\cite[Lemma~2]{kurdyka1998gradients} to $\Psi'$, and shrinking $\rho$
if necessary, there exists $\bar{\eta}\in(0,\rho)$ such that
$\Psi'$ is monotone on $(0,\bar{\eta})$. $\Psi$ is strictly increasing and $C^1$, so $\Psi'(s) \geq 0$ for all $s>0$. Moreover, $\Psi'(s) > 0$ for all $s \in (0,\bar{\eta})$. Indeed, if $\Psi'(s_0) = 0$ for some $s_0 \in (0,\bar{\eta})$, then $\Psi'$ vanishes either on $(0,s_0]$ or on $[s_0,\bar \eta)$, a clear contradiction. 

Suppose that $\Psi'$ is nonincreasing. Set $\eta\coloneqq\bar{\eta}$, $\varphi(s)\coloneqq
\tfrac{\Psi(s)-\Psi(0)}{c}$ for $ s\in[0,\eta).$ Dividing \cref{eq:kurdyka_intermediate} by $c$ gives $\varphi'\bigl(f(x)-f(x^*)\bigr)\|\nabla f(x)\| \geq 1,$ and $\varphi$ clearly satisfies \labelcref{def:KL_i,def:KL_ii,def:KL_iii,def:KL_iv} of~\cref{def:KL}.

Now suppose that $\Psi'$ is nondecreasing. Set $\eta\coloneqq\bar{\eta}/2$, $M\coloneqq \Psi'(\eta)>0$,
and define $ \varphi(s) \coloneqq \tfrac{M}{c}s$, for $ s\in[0,\eta)$. $\varphi$ is linear and therefore concave. Since $\Psi'$ is nondecreasing, $\Psi'(s)\leq M$ for every $s\in (0,\eta)$. Therefore, whenever $x\in U$ and $f(x^*)<f(x)<f(x^*)+\eta$,
we have
\begin{align}
    \varphi'\bigl(f(x)-f(x^*)\bigr)\|\nabla f(x)\| = 
    \frac{M}{c}\|\nabla f(x)\|  \geq
    \frac{\Psi'\bigl(f(x)-f(x^*)\bigr)}{c}
    \|\nabla f(x)\| \geq 1.
\end{align}
Conditions (i) through (iii) of~\cref{def:KL} are straightforward to verify.

Thus, in either case, $f$ satisfies the KL property at $x^*$. Since $x^*\in\R^n$ was arbitrary, $f$ is a KL function.
\end{proof}

\begin{lemma}[Definability of the implicit LPN regularizer]\label{lem:definability}
Let $\mathcal M$ be an o-minimal structure that includes polynomials, and suppose that $\lpn$ is an LPN, with potential $\potential$, implicit regularizer $\reg$, and activation $\sigma$. If $\sigma$ is $\mathcal M$-definable, then $\reg$ is $\mathcal M$-definable.
\end{lemma}
\begin{proof}
Since the activation is $\mathcal M$-definable, and definability is preserved through sums, products, and composition, $\potential$ is $\mathcal M$-definable. Recall that $\potential$ is strongly convex. That means for each $y$, $x \mapsto \potential(x) - \langle x,y\rangle$ is strongly convex and coercive, so it attains a unique minimum. Equivalently, $\potential^*(y) = \max_{x\in \R^n}\left \{\langle x,y\rangle - \potential(x) \right\}.$
We can therefore characterize its graph directly. $(y,t) \in \graph \potential^*$ if and only if there exists $x \in \R^n$ such that $t=\langle x,y\rangle - \potential(x)$ and for all $z\in \R^n$, we have $\langle z,y\rangle -\potential(z)\leq t$. Define \begin{equation}
    A=\{(y,t,x)\mid t=\langle x,y\rangle - \potential(x)\}, \quad B = \{(y,t,z) \mid \langle y,z\rangle -\potential(z) > t\}.
\end{equation}
Both are definable. Let $\pi$ denote projection onto the $(y,t)$-coordinates, an operation that preserves $\mathcal M$-definability. Finally, since $\mathcal{M}$-definable sets are closed under complements and finite intersections, 
\begin{align}
\operatorname{graph}\potential^*
&=
\left\{
(y,t):
\exists x,\;
t=\langle x,y\rangle-\potential(x)
\right\}
\cap
\left(
\left\{
(y,t):
\exists z,\;
\langle z,y\rangle-\potential(z)>t
\right\}
\right)^c
\\
&= \pi(A) \cap \pi(B)^c,
\end{align}
making $\potential^*$ $\mathcal{M}$-definable. Definability of $\reg$ follows since polynomials are $\mathcal M$-definable.
\end{proof}

\begin{proof}[Proof of~\cref{prop:lpn_objective_KL}]
Let $\Ranexp$ denote the expansion of the real field, $(\R,+,\cdot,<,0,1)$, by all restricted real-analytic functions together with the exponential function. $\Ranexp$ is o-minimal by~\cite[Corollary 5.13]{van1994elementary}. Further, polynomials are definable in any expansion of the real field. Also, $\graph \log =\left\{(x,y)\in(0,\infty)\times\R:(y,x)\in\graph \exp\right\}$,
so $\log$ is definable. Note that the absolute value can be written as $|x| = \max\{x,-x\}$. Since maxima, minima, sums, products, and compositions preserve $\Ranexp$-definability,\footnote{See ~\cite[B.1]{ji2020directional}, for example.} every activation $\sigma$ and data-fidelity term $\fid$ satisfying the assumptions of~\cref{prop:lpn_objective_KL} is $\Ranexp$-definable. By~\cref{lem:definability}, $\reg$ is definable. Therefore, for any $\eta>0,\gamma>0$, $F = \eta \fid + \gamma \reg$ is definable. Observe that $\reg$ is $C^1$ by~\cref{lem:C_1_regularizer}, so $F$ is KL by~\cref{lem:definable_KL}. Next, for the augmented objective, $t\mapsto \frac{1}{2}t^2$ is definable. Hence, $\xi$ is definable, and therefore KL by the same result. Since  $\lpn^{(\lambda)}$ is an LPN by~\cref{prop:closure_of_lpns}, its implicit regularizer, which is $\lambda M_{1-\lambda}\reg$ by~\cref{cor:averaged_lpn}, is $\Ranexp$-definable. $F^{(\lambda)}$ is therefore KL. 
\end{proof}


\section{Other Optimization Methods}\label{section:other_methods}
In this section, we discuss the proximal point method (PPM) and the alternating direction method of multipliers (ADMM). We begin with PPM, which, for an objective $\varphi:\R^n\to\R$, repeatedly applies $\prox_\varphi$ to the iterates. In the LPN setting,  this algorithm amounts to repeatedly applying the LPN. PPM is already a special case of the LPN-PGD framework developed in~\cref{section:convergence_analysis}. Taking the data-fidelity term to be identically zero, $\fid$ is $0$-smooth and $\nabla\fid\equiv0$, so the LPN-PGD iteration in~\cref{alg:lpn_pgd_vanilla} reduces to
\begin{equation}
    x_{k+1}
    = \lpn \left(x_k-\eta\nabla\fid(x_k)\right)
    = \lpn(x_k)
    = \prox_{\reg}(x_k)
\end{equation}
Consequently, the convergence results of~\cref{prop:pgd_convergence} apply directly to PPM. In particular, since $\nabla\fid\equiv0$, there is no step-size parameter or corresponding smoothness restriction.

We next consider ADMM using a proximal operator, described in \cref{alg:lpn_admm}. LPN-ADMM has the same requirements as LPN-PGD, with the additional assumption that $\fid$ is convex. Thus, LPN-ADMM applies to a strictly narrower class of problems, while providing the same guarantees obtained for LPN-PGD. Moreover, each ADMM iteration requires evaluating $\prox_{\eta\fid}$, whereas LPN-PGD requires only the gradient $\nabla\fid$, making ADMM potentially more complicated to implement. For these reasons, we do not present extensions of LPN-ADMM analogous to the averaged and scaled LPN variants developed for LPN-PGD.
\begin{algorithm}[H]
\caption{LPN-ADMM}
\label{alg:lpn_admm}
\begin{algorithmic}[1]
\REQUIRE Initial point $x_0\in\R^n$, $L_f$-smooth, convex, and bounded below $\fid:\R^n\to\R$, step size $0<\eta<1/L_f$, LPN $\lpn$ as in~\cref{def:LPN} with potential $\potential$, and number of steps $K$.
\STATE $u_0\coloneqq -\eta\nabla\fid(x_0)$
\STATE $v_0\coloneqq \lpn(x_0+u_0)$
\FOR{$k=0,1,\dots,K-1$}
    \STATE $x_{k+1}\coloneqq \prox_{\eta\fid}(v_k-u_k)$
    \STATE $u_{k+1}\coloneqq u_k+x_{k+1}-v_k$
    \STATE $v_{k+1}\coloneqq \lpn(x_{k+1}+u_{k+1})$
\ENDFOR
\RETURN $x_K$
\end{algorithmic}
\end{algorithm}
\begin{proposition}[Convergence of LPN-ADMM]\label{prop:admm_convergence}
Consider $\{x_k\}_{k \geq 0}$, the iterates of standard LPN-ADMM in~\cref{alg:lpn_admm}, where $\reg$ denotes the regularizer associated with $\lpn$. Letting $F \coloneqq \eta \fid + \reg$ and \begin{equation}
    \Loss(x,v,u) \coloneqq \eta \fid(x) + \reg(v) + \langle u, x-v\rangle + \frac{1}{2}\|x-v\|^2, \quad \Loss_k \coloneqq \Loss(x_k,v_k,u_k),
\end{equation}
the following hold:
\begin{enumerate}[
    label=\normalfont(\roman*),
    ref=\thetheorem(\roman*)]
\crefalias{enumi}{proposition}
    \item
    $\Loss_k$ is non-increasing and convergent.

    \item 
    $\sum_{k=0}^{+\infty}
    (\|x_k-v_k\|^2+\|x_{k+1}-x_k\|^2)<+\infty$.

    \item
    If $\Loss$ is KL, then
    $\sum_{k=1}^{+\infty}
    (\|v_{k+1}-v_k\|+\|x_{k+1}-v_{k+1}\|)<+\infty$,
    and $\{(v_k,x_k,u_k)\}_{k\geq1}$ converges to a critical point of $\Loss$.

\end{enumerate}
\end{proposition}
\begin{proof}
    By the optimality of the $x$-update, $\eta \nabla\fid(x_{k+1}) + x_{k+1} - v_k+u_k=0$. From the definition of $u_{k+1}$, we get \begin{equation}\label{eq:admm_proof1}u_k = -\eta \nabla \fid (x_k), \quad \forall k \geq 0,\end{equation}
    since $u_0 \coloneqq -\eta \nabla \fid(x_0)$. Similarly, since $\lpn (x_{k+1}+u_{k+1}) \in \prox_{\reg}(x_{k+1}+u_{k+1})$, optimality of the $v$-update gives $\nabla \reg(v_{k+1}) = u_{k+1} + x_{k+1} - v_{k+1} = -\eta \nabla \fid(x_{k+1}) + x_{k+1} - v_{k+1}.$ We can conclude that $\nabla F(v_k) = \eta \nabla \fid (v_k) - \eta \nabla \fid (x_k) + x_k - v_k$, meaning $
        \|\nabla F(v_k)\| \leq (1+L)\|x_k - v_k\|.$ 
    Let $d_k \coloneqq x_k - x_{k+1}$, $g_k \coloneqq \eta \nabla \fid (x_k) - \eta \nabla \fid(x_{k+1})$. Then we have \begin{equation}\label{eq:admm_proof4}
        x_{k+1} - v_k =u_{k+1}-u_k= g_k, \quad x_k - v_k = d_k + g_k.
    \end{equation}
    Since $v_{k+1}$ minimizes $v \mapsto \Loss(x_{k+1},v, u_{k+1})$, we have \begin{equation}\label{eq:admm_proof5}
    \Loss_{k+1} \leq \Loss(x_{k+1},v_k, u_{k+1}). 
    \end{equation}
    From~\cref{eq:admm_proof1} and~\cref{eq:admm_proof4},
    \begin{equation}\label{eq:admm_proof11}
        \Loss_k - \Loss(x_{k+1},v_k, u_{k+1}) = \eta \fid(x_k) - \eta \fid(x_{k+1})-\langle \eta \nabla \fid(x_{k+1}),d_k\rangle + \frac{1}{2}\|d_k\|^2 - \|g_k\|^2. 
    \end{equation}
    Since $\eta \fid$ is convex and $L$-smooth for some $L<1$, \begin{equation}\label{eq:admm_proof6}
    \eta \fid(x_k) \geq \eta \fid(x_{k+1}) + \langle \eta \nabla \fid (x_{k+1}), d_k\rangle + \frac{1}{2L}\|g_k\|^2.    
    \end{equation}
    Combining~\cref{eq:admm_proof5,eq:admm_proof11,eq:admm_proof6}, $ \Loss_k - \Loss_{k+1} \geq \Loss_k - \Loss(x_{k+1},v_k,u_{k+1}) \geq \frac{1}{2}\|d_k\|^2 + \left( \frac{1}{2L}-1\right) \|g_k\|^2. $
    If $0<L\leq 1/2$, the second term is nonnegative, and therefore $\Loss_k-\Loss_{k+1} \geq \tfrac{1}{2}\|d_k\|^2$. If $1/2 < L < 1$, then $\tfrac{1}{2L}-1 < 0$. Since $\eta \fid$ is $L$-smooth, $\|g_k\| \leq L\|d_k\|$, and hence \begin{equation}
        \left(\frac{1}{2L}-1\right)\|g_k\|^2 \geq \left(\frac{L}{2}-L^2\right)\|d_k\|^2.
    \end{equation}
    Therefore, $\Loss_k - \Loss_{k+1} \geq \left(\frac{1}{2}+\frac{L}{2} - L^2\right)\|d_k\|^2,$ and we can conclude that \begin{equation}\label{eq:admm_proof8}
        \Loss_k - \Loss_{k+1} \geq C_L \|x_k - x_{k+1}\|^2, \quad C_L \coloneqq \begin{cases}
            \frac{1}{2}, & 0 < L \leq \frac{1}{2},\\ \frac{1}{2}+\frac{L}{2}-L^2 , & \frac{1}{2}< L < 1.
        \end{cases}
    \end{equation}
    Clearly $C_L > 0$ whenever $0 < L < 1$. Next,
    \begin{equation}
        \Loss_k = \eta \fid (x_k) + \reg(v_k) - \langle \eta \nabla \fid(x_k),x_k-v_k\rangle + \frac{1}{2}\|x_k-v_k\|^2.
    \end{equation}
    Since $\eta \fid $ is $L$-smooth, 
    \begin{equation} \label{eq:admm_proof10}
        \Loss_k \geq F(v_k) + \frac{1-L}{2}\|x_k-v_k\|^2.
    \end{equation}
    And in particular $\Loss_k \geq \inf F$. Combined with~\cref{eq:admm_proof8}, there exists some number $\Loss_\infty$ so that $\Loss_k \downarrow \Loss_\infty$, which is (i). 
Next, from~\cref{eq:admm_proof4} and $L$-smoothness, $\|x_k-v_k\|\leq \|d_k\|+\|g_k\| \leq (1+L)\|d_k\|.$ Hence, $\|x_k-v_k\|^2 \leq (1+L)^2 \|x_k-x_{k+1}\|^2.$ Summing~\cref{eq:admm_proof8} from $k=0$ to $k=K$ gives $C_L \sum_{k=0}^K \|x_k -x_{k+1}\|^2 \leq \Loss_0 - \Loss_{k+1}.$ Combining with the previous inequality yields \begin{equation}
        \sum_{k=0}^K (\|x_k-v_k\|^2 + \|x_k-x_{k+1}\|^2) \leq \frac{1+(1+L)^2}{C_L}(\Loss_0-\Loss_{K+1}).
    \end{equation}
    Taking $K\to \infty $ gives us (ii). 

    Next, let $z_k \coloneqq (x_k,v_k,u_k)$, and assume $\Loss$ satisfies the KL property. By (ii), $\|d_k\|\to 0$ and $\|x_k-v_k\| \to 0$. From~\cref{eq:admm_proof1} and the optimality of the $v$-update, we have
    \begin{align}
        \nabla _x \Loss(x_k,v_k,u_k) &= \eta \nabla \fid(x_k) + u_k + x_k - v_k = x_k - v_k,\nonumber \\
        \nabla_v \Loss(x_k,v_k,u_k)  &= \nabla \reg(v_k) - u_k -x_k + v_k = 0,\\
        \nabla_u \Loss(x_k,v_k,u_k)  &=x_k - v_k\nonumber.
    \end{align}
    Since $x_k-v_k = x_k - x_{k+1} + \eta \nabla \fid(x_k) - \eta \nabla \fid(x_{k+1}),$ \begin{equation}\label{eq:admm_proof12}
         \|\nabla \Loss(z_k)\| =\sqrt{2}\|x_k-v_k\|\leq \sqrt{2}(1+L)\|d_k\| .
     \end{equation} By~\cref{eq:admm_proof10} and the monotonicity of $\Loss_k$, $F(v_k) \leq \Loss_k \leq \Loss_0$. Since $F$ is coercive, $\{v_k\}$ is bounded. \cref{eq:admm_proof10} also implies $\{x_k-v_k\}$, and hence $\{x_k\}$ are bounded. Finally,  $\{u_k\}$ is bounded by~\cref{eq:admm_proof1}, since $\nabla \fid$ is continuous. Therefore $\{z_k\}$ is bounded, and the set of cluster points, which we denote $\Omega$ is nonempty and compact. By (i), $\Loss_k \to \Loss_\infty$. Therefore, if $z^* \in \Omega$, there exists a subsequence $z_{k_j} \to z^*$, and continuity of $\Loss$ means $\Loss(z^*) = \lim_{j\to \infty}\Loss(z_{k_j}) = \Loss_\infty.$
     Since $\|d_k\| \to 0$, and by~\cref{eq:admm_proof12}, $\|\nabla\Loss(z_k)\| \to 0$, meaning $\nabla \Loss(z^*) = 0$. Thus every point of $\Omega$ is a critical point of $\Loss$ and $\Loss$ takes the value $\Loss_\infty$ on $\Omega$.

    By the KL property, for every $z^* \in \Omega$, there is a neighborhood $U_{z^*}$, some $\eta_{z^*} > 0$ and a concave $C^1$ desingularizing function $\varphi_{z^*}$ such that $\varphi_{z^*}'(\Loss(z)-\Loss_\infty)\|\nabla \Loss(z)\| \geq 1,$
     whenever $z \in U_{z^*}$ and $0 < \Loss(z) - \Loss_\infty < \eta_{z^*}$. $\Omega$ is compact, so we may cover it with finitely many of these neighborhoods, call them $U_1,...,U_m$, let $\bar \eta = \min_{1 \leq i\leq m}\eta_i$, and $\varphi(s) = \sum_{i=1}^m \varphi_i(s)$ for $0 \leq s < \bar \eta$. This $\varphi$ is still concave, $C^1$ on $(0,\bar \eta)$, and has a positive derivative. Further, when $z \in U_i$, $\varphi'(\Loss(z)-\Loss_\infty)\|\nabla \Loss(z)\| \geq 1,$
     Clearly $\mathrm{dist}(z_k ,\Omega) \to 0$, so $z_k \in U_1\cup...\cup U_m$ for sufficiently large $k$. Since $\Loss_k \to \Loss_\infty$, the KL inequality applies for sufficiently large $k$, i.e., $\varphi'(\Loss_k - \Loss_\infty) \geq \frac{1}{\|\nabla \Loss(z_k)\|}$. By concavity of $\varphi$, $\varphi(\Loss_k-\Loss_\infty)- \varphi(\Loss_{k+1}-\Loss_\infty) \geq \varphi'(\Loss_k-\Loss_\infty)(\Loss_k - \Loss_{k+1}),$ and from the KL inequality,~\cref{eq:admm_proof8}, and~\cref{eq:admm_proof12},
     \begin{equation}
         \varphi(\Loss_k - \Loss_\infty) - \varphi(\Loss_{k+1}-\Loss_\infty) \geq \frac{C_L \|d_k\|^2}{\|\nabla \Loss(z_k)\|} \geq \frac{C_L}{\sqrt{2}(1+L)}\|d_k\|.
     \end{equation}
     Summing over all sufficiently large $k$, we get $\sum_{k=0}^\infty\|x_{k+1}-x_k\| < \infty$. Since $\|x_k-v_k\| \leq (1+L)\|d_k\|$, we have $\sum_{k=0}^\infty \|x_k - v_k\| < \infty$. Next, $u_{k+1}-u_k = g_k$ and $\|g_k\| \leq L\|d_k\|$, so $\sum_{k=0}^\infty \|u_{k+1}-u_k\| < \infty$. Finally, since $v_k = x_{k+1}-g_k$, $\|v_{k+1}-v_k\| \leq \|x_{k+2}-x_{k+1}\| + \|g_{k+1}\| + \|g_k\| \leq (1+L) \|d_{k+1}\| + L\|d_k\|$. Therefore, $\sum_{k=0}^\infty \| v_{k+1}-v_k\| < \infty$. Thus all three sequences $\{x_k\},\{v_k\}$, and $\{u_k\}$ have finite length and converge. Therefore $z_k \to z^*$ for some $z^*\in \Omega$, and the result follows. 
\end{proof}

\putbib[ref]
\end{bibunit}

\end{document}